\documentclass[final,leqno,onefignum,onetabnum]{siamltex1213}

\usepackage{amsmath, amssymb}
\usepackage{mathtools}   % need for `show only references'
\mathtoolsset{showonlyrefs=true}  % only equations which are labeled AND referenced will be numbered.

\allowdisplaybreaks

\newtheorem{remark}{Remark}[section]
\newtheorem{algorithm}{Algorithm}

\newcommand{\E}{\mathbb{E}}
\newcommand \Esp[1]{\E\left(#1\right)}
\newcommand \Espi[1]{\E_i\left(#1\right)}
\newcommand \EspMik[1]{\E^{(M)}_{i,k}\left(#1\right)}
\newcommand \plog{p_{\rm logis.}^{(\mu)}}

\renewcommand{\L}{\mathbf{L}}
\renewcommand{\P}{\mathbb{P}}
\newcommand{\R}{\mathbb{R}}
\newcommand{\N}{\mathbb{N}}
\newcommand{\1}{\mathbf{1}}

\newcommand{\cB}{\mathcal{B}}
\newcommand{\cC}{\mathcal{C}}
\newcommand{\cE}{\mathcal{E}}
\newcommand{\cF}{\mathcal{F}}

\newcommand{\cH}{\mathcal{H}}
\newcommand{\cK}{\mathcal{K}}
\newcommand{\cL}{\mathcal{L}}
\newcommand{\cM}{\mathcal{M}}

\newcommand{\cT}{\mathcal{T}}
\newcommand{\cX}{\mathcal{X}}

\newcommand{\Om}[1]{\Omega^{(#1)} }
\newcommand{\F}[2]{\cF^{(#1)}_{#2}}

\newcommand{\X}[2]{X^{(#1)}_{#2}}
\newcommand{\vecX}[2]{{\underline{\mathbf X}}^{(#1)}_{#2}}

\newcommand{\DW}[2]{\Delta W^{(#1)}_{#2}}
\newcommand{\yM}[2]{y^{(M)}_{#1}(#2)}
\newcommand{\zM}[2]{z^{(M)}_{#1}(#2)}
\newcommand{\ObsM}[2]{S^{(M)}_{#1}(#2)}
\newcommand{\Obs}[2]{S_{#1}(#2)}

\newcommand{\xiM }[2]{\xi^*_{#1}(#2)}

\newcommand{\PsiM}[2]{\psi_{#1}(#2)}
\newcommand{\psim}{\psi^{(M)}}

\newcommand \Dt{\Delta_t}

\newcommand{\LinSpace}[1]{\cL_{#1}}
\newcommand{\DimLinSpace}[1]{{\rm dim}(\LinSpace{#1})}
\newcommand{\nuDimLinSpace}[1]{\nu(\DimLinSpace{#1})}
\newcommand \xkm[1]{x^{-}_{k,#1}}
\newcommand \xkp[1]{x^{+}_{k,#1}}
\newcommand \sumk{\sum_{k=1}^K}
\newcommand \sumji{\sum_{j=i}^{N-1}}
\newcommand{\OLS}{{\bf OLS}}
\newcommand{\LPz}{{\bf LP0}}
\newcommand{\LPo}{{\bf LP1}}

\newcommand{\HF}{{$\bf (A_f)$}}
\newcommand{\HG}{{$\bf (A_g)$}}
\newcommand{\Hnu}{{$\bf (A_\nu)$}}
\newcommand{\HX}{{$\bf (A_X)$}}
\newcommand{\cnu}{c_{\text{\Hnu}}}
\newcommand{\HS}{{$\bf (A_{\rm Strat.})$}}

\newcommand{\cprop}{C_{\ref{prop: apriori}}}
\newcommand{\cgamma}{ C_{\ref{thm:mc:err decomp} } }
\newcommand {\Cas}{C_{\ref{cor: as 2}} }

\newcommand{\empinorm}[2]{\left| #1\right|_{#2,M} }

\newcommand{\normnu}[1]{\left|#1 \right|_{\nu}}
\newcommand{\normnuk}[1]{\left|#1 \right|_{\nu_k}}

\newcommand \xnuz{X^{0,\nu}}

\newcommand{\vecx}[1]{\mathbf{\underline{#1} } } %Bold vector
\newcommand{\dd}{{\rm d}}
\newcommand{\dx}{{\rm d}x}
\newcommand \qed{\hfill$\Box$}

\usepackage{color}
\definecolor{gray97}{gray}{.97}
\definecolor{gray75}{gray}{.75}
\definecolor{gray45}{gray}{.45}
\definecolor{gray35}{gray}{.35}

\newcommand \LoadF{{\cal C}_{\rm Load\ factor}}

\usepackage{listings}
\lstnewenvironment{listing}[1][]
   {\lstset{#1}\pagebreak[0]}{\pagebreak[0]}

\lstdefinestyle{consola}
   {basicstyle=\scriptsize\bf\ttfamily,
    backgroundcolor=\color{gray75},
   }

\lstdefinestyle{C}
   {language=C,
   }

\title{Stratified regression Monte-Carlo scheme for semilinear PDEs and BSDEs with large scale parallelization on GPUs}

\author{E. Gobet\thanks{Centre de Math\'ematiques Appliqu\'ees, Ecole Polytechnique and CNRS, route de Saclay, 91128 Palaiseau cedex, France. This research is part of the Chair Financial Risks of the Risk Foundation, the FiME Laboratory and the ANR project CAESARS (ANR-15-CE05-0024). 
Email: {\tt emmanuel.gobet@polytechnique.edu}}
\and J. G. L\'opez-Salas\thanks{Department of Mathematics, Faculty of Informatics, Universidade da Coru\~na, Campus de Elvi\~na s/n, 15071 - A Coru\~na, Spain. This author has been partially funded by Spanish Grant MTM2013-47800-C2-1-P and also by a Spanish FPU grant. Email: {\tt jose.lsalas@udc.es}}
\and P. Turkedjiev\thanks{King's College London, Department of Mathematics, Strand, London WC2R 2LS, United Kingdom.
This research is supported the Chair Financial Risks of the Risk Foundation and of the FiME Laboratory.
Email: {\tt plamen.turkedjiev@kcl.ac.uk}}
\and C. V\'azquez\thanks{Department of Mathematics, Faculty of Informatics, Universidade da Coru\~na, Campus de Elvi\~na s/n, 15071 - A Coru\~na, Spain. This author has been partially funded by Spanish Grant MTM2013-47800-C2-1-P. Email: {\tt carlosv@udc.es}}
}

\begin{document}
\maketitle
\slugger{sisc}{xxxx}{xx}{x}{x--x}%slugger should be set to mms, siap, sicomp, sicon, sidma, sima, simax, sinum, siopt, sisc, or sirev

\begin{abstract}
In this paper, we design a novel algorithm based on  Least-Squares Monte Carlo (LSMC) in order to approximate the solution of discrete time Backward Stochastic Differential Equations (BSDEs). Our algorithm allows massive parallelization of the computations on many core processors such as  graphics processing units (GPUs). Our approach consists of a novel method of stratification which appears to be crucial for large scale parallelization. In this way, we minimize the exposure to the memory requirements due to the storage of simulations. Indeed, we note the lower memory overhead of the method compared with previous works.
\end{abstract}



\pagestyle{myheadings}
\thispagestyle{plain}
\markboth{E. Gobet, J. G. L\'opez-Salas, P. Turkedjiev and C. V\'azquez}{Stratified regression Monte-Carlo scheme for BSDEs}

\section{Introduction}

\paragraph{The problem}
The aim of the algorithm in this paper is to approximate the $(Y,Z)$ components of the solution to the decoupled forward-backward stochastic differential equation (BSDE)
\begin{align}
%\begin{split}
\label{eq:fbsde}
Y_t& =  g(X_T) + \int_t^T f(s,X_s,Y_s,Z_s) \dd s - \int_t^T Z_s \dd W_s, \\
\label{eq:eds}
X_t & =  x + \int_{ 0} ^t b(s,X_s) \dd s + \int_{0} ^t \sigma(s,X_s) \dd W_s,
%\end{split}
\end{align}
where $W$ is a $q\ge 1$ dimensional Brownian motion.
The algorithm will also approximate the solution $u$ to the related semilinear, parabolic partial differential equation (PDE) of the form
\begin{equation}
\label{eq:pde}
\partial_{t}u(t,x)+{\cal A}u(t,x)+f(t,x,u(t,x),{\nabla_x} u\sigma(t,x))=0\ \text{ for $t<T$ and } u(T,.)=g(.),
\end{equation}
where $\cal A$ is the infinitesimal generator of $X$, through the Feynman-Kac relation $(Y_{t},Z_{t})=(u(t,X_t),(\nabla_{x} u\sigma) (t,X_{t}))$.
In recent times, there has been an increasing interest to have algorithms which work efficiently when the dimension $d$ of the space occupied by the process $X$ is large.
This interest has been principally driven by the mathematical finance community, where nonlinear valuation rules are becoming increasingly important.

In general, currently available algorithms \cite{gobe:laba:10,bend:stei:12,bouc:wari:12,laba:lelo:13,bria:laba:14,gobe:turk:13b,gobe:turk:14,gobe:turk:15} rarely handle the case of dimension greater than 8.
The main constraint is not only due to the computational time, but mainly due to memory consumption requirements by the algorithms.
For example, the recent work \cite{gobe:turk:14} uses a Regression Monte Carlo approach (a.k.a. Least Squares MC), in which the solutions $(u,\nabla_{x} u\sigma)$ of the semi-linear PDE are approximated on a $\cK$-dimensional basis of functions at each point of a time grid of cardinality $N$. Popular choices of basis functions are global  \cite{laba:lelo:13} or local polynomials \cite{gobe:turk:14}. In both cases, the approximation error behaves in general like $\cK ^{-\alpha'/d}$ where $\alpha'$ measures the smoothness of the function of interest and $d$ is the dimension (curse of dimensionality): see \cite[Theorem 6.2.6]{funa:92} for global polynomials, see \cite[Section 11.2]{gyor:kohl:krzy:walk:02} for local polynomials. 
Later, we  use local approximations in order to allow parallel computing. We restrict to affine polynomials for  implementation in GPU.  The coefficients of the basis functions are computed at every time point $t_i$ with the aid of $\cM$ simulations of a discrete time Markov chain (which approximates $X$) in the interval $[t_i,T]$.
The main memory constraints of this scheme are (a) to store the $\cK \times N$ coefficients of the basis functions, and (b) to store the $\cM\times N$ simulations used to compute the coefficients.
To illustrate the problem of high dimension, in order to ensure the convergence the dimension of the basis is typically $\cK = { const}\times N^{\alpha d}$, for some $\alpha >0$ (which decreases with the regularity of the solution),
so $\cK$ increases geometrically with $d$.
Moreover, the error analysis of these algorithms demonstrates that the local statistical error
is proportional to $N\cK/\cM$, so that one must choose $\cM = {const}\times \cK N^2$ to ensure a convergence  $O(N^{-1})$ of the scheme.
This implies that the simulations pose by far the most significant constraint on the memory.

\paragraph{Objectives}
The purpose of this paper is to drastically rework the algorithm of \cite{gobe:turk:14} to first minimize the exposure to the memory due to the storage of simulations. This will allow computation in larger dimension $d$. Secondly, in this way the algorithm can be implemented in parallel on GPU processors to optimize the computational time.

\paragraph{New Regression Monte Carlo paradigm}
We develop a novel algorithm called the Stratified Regression MDP (SRMDP) algorithm; the name is aimed to distinguish from the related  LSMDP algorithm \cite{gobe:turk:14}.
The key technique is to use {\em stratified simulation} of the paths of $X$.
In order to estimate the solution at $t_i$,
we first define a set of hypercubes $(\cH_k\subset\R^d:1\leq k \leq K)$.
Then, for each hypercube $\cH_k$, we simulate $M$ paths of the process $X$ in the interval $[t_i,T]$ starting from i.i.d. random variables valued in $\cH_k$;
these random variables are distributed according to the conditional logistic distribution, see \Hnu\ later. 
By using only the paths starting in $\cH_k$, we approximate the  solution to the BSDE restricted to $X_{t_{i}}\in\cH_k$ on linear functions spaces $\LinSpace{Y,k}$ and $\LinSpace{Z,k}$ (both of small dimension), see \HS\ later. 
\footnote{To distinguish from previous algorithms, we use two notations  for the number of simulations in this section: $\cal M$ and $M$. $\cal M$ stands for the overall number of simulations for computing the full approximation in the unstratified algorithms, while $M$ stands  for the number of simulations used to evaluate the approximation locally in each stratum (our stratified regression algorithm). Later we will mainly use $M$.}
This allows us to minimize the amount of memory consumed by the simulations, since we only need to generate samples on one hypercube at a time. 
In Theorem \ref{thm:mc:err decomp}, we demonstrate that the error of our scheme is proportional to $N\max(\DimLinSpace{Y,k},\DimLinSpace{Z,k})/M$ and, since $\max(\DimLinSpace{Y,k},\DimLinSpace{Z,k}) = const$, we require only $M = const \times N^2 $ to ensure the convergence $O(N^{-1})$.
Therefore, the memory consumption of the algorithm will be dominated by the storage of the coefficients, which equals $const \times N^{\alpha d}$  (the theoretical minimum).
Moreover, the computations are performed in parallel across the hypercubes, which allows for massive parallelization.
The speed-up compared to sequential programming increases as the dimension $d$ increases, because of the geometric growth of the number of hypercubes with respect to $d$.
%in high dimension $d$ due to the geometric growth of the number of hypercubes with respect to $d$.
In the subsequent tests (\S \ref{section:numerics}), for instance we can solve problems in dimension $d=11$ within eight seconds using 2000 simulations per hypercube.

This regression Monte Carlo approach is very different from the algorithm proposed in \cite{gobe:turk:14}.
Although local approximations were already proposed in that work,
the paths of the process $X$ were simulated from a fixed point at time 0 rather than directly in the hypercubes.
This implies that one must store {\em all}  the simulated paths at any given time, rather than only those for the specific hypercubes.
This is because the trajectories  are random, and one is not certain which paths will end up in which hypercubes a priori.
Therefore, our scheme essentially removes the main constraint on the memory consumption of LSMC algorithms for BSDEs.

The choice of the logistic distribution for the stratification procedure is crucial.
Firstly, it is easy to simulate from the conditional distribution.
Secondly, it possesses the important USES property (see later \Hnu), which enables us to recover equivalent $\L_{2}$-norms (up to constant) for the marginal of the forward process initialized with the logistic distribution (Proposition \ref{prop:hnu}).

\paragraph{Literature review} Parallelization of Monte-Carlo methods for solving non-linear probabilistic equations has been little investigated. Due to the non-linearity, this is a challenging issue. For optimal stopping problems, we can refer to the works \cite{abba:lape:09,abba:vial:lape:merc:09,fati:phil:13} with numerical results up to dimension 4. To the best of our knowledge, the only work related to BSDEs in parallel version is \cite{laba:lelo:13}. It is based on a Picard iteration for finding the solution, coupled with iterative control variates. The iterative solution is computed through an approximation on sparse polynomial basis. Although the authors report efficient numerical experiments up to dimension 8, this study is not supported by a theoretical error analysis. Due to the stratification, our proposed approach is quite different from  \cite{laba:lelo:13} and additionally, we provide an error analysis (Theorem \ref{thm:mc:err decomp}).
%\todo{Perhaps some more on GPU: how is it usually used in math finance, why is it awkward for nonlinear problems?}

\paragraph{Notation}
\begin{romannum}
\item $|x|$ stands for the Euclidean norm of the vector $x$.
\item $\log(x)$ stands for the natural logarithm of $x\in\R_+$.
\item For a multidimensional process $U=(U_i)_{0\leq i\leq N}$, its $l$-th component is denoted by $U_l=(U_{l,i})_{0\leq i\leq N}$.
\item For any finite $L>0$ and $x = (x_1,\dots,x_n)\in \R^n$, define the truncation function
\begin{equation}
\label{eq:TL}
\cT_L(x) := (-L\vee x_1 \wedge L, \dots, -L \vee x_n \wedge L).
\end{equation}

\item For a probability measure $\nu$ on a domain $D$, and function $h:D \to \R^l$ in $\L_2(D,\nu)$, denote the $\L_2$ norm of $h$ by $\normnu{h} := \sqrt{\int_D |h|^2(x) \nu(\dx)}$.

\item For  a probability measure $\nu$,  disjoint sets $\{\cH_1,\ldots,\cH_K\}$ in the support of $\nu$, and finite dimensional function spaces $\LinSpace {} \{\LinSpace{1},\ldots,\LinSpace{K}\}$ such that the domain of $\LinSpace k$ is in the respected set  $\cH_{k}$
\begin{equation}
\label{eq:def:nu:dim:linspace}
\nuDimLinSpace{ }=\sumk\nu(\cH_k)\DimLinSpace{k}.
\end{equation}

\item For function $g: \R_+ \to \R_+$, the order notation $g(x) = O(x)$ means that there exists some universal unspecified constant, $const >0$, such that $g(x) \leq const \times x$ for all $x\in \R_+$.
\end{romannum}

\section{Mathematical framework and basic properties}
We work on a filtered probability space $(\Omega, \mathcal{F}, (\mathcal{F}_t)_{0\le t\le T},\mathbb{P})$ containing a $q$-dimensional ($q\geq 1$) Brownian motion $W$.
The filtration $(\mathcal{F}_t)_{0\le t\le T}$  satisfies the usual hypotheses. The existence of a unique strong 
solution $X$ to the forward equation  \eqref{eq:eds} follows from usual Lipschitz conditions on $b$ and $\sigma$, see \HX.
The BSDE \eqref{eq:fbsde} is approximated using a multistep-forward dynamical programming equation (MDP) studied in \cite{gobe:turk:13b}.
Let $\pi := \{t_i:=i \Dt:0\leq i \leq N \}$ be the uniform time-grid with time step $\Dt=T/N$. The solution $(Y_i,Z_i)_{0\leq i \leq N-1}$ of the MDP  can be written in the form:
{\small 
\begin{align}
\label{eq:MDP:intro}
\left.
\begin{array}{rcl}
Y_i&=&\Espi{ g(X_N) +\sum_{j=i}^{N-1}f_j(X_j,Y_{j+1},Z_j) \Dt},\\
\Dt Z_i&=&\Espi{ (g(X_N) +\sum_{j=i+1}^{N-1} f_j(X_j,Y_{j+1},Z_j) \Dt ) \Delta W_i}
\end{array}
\right\}
\quad \text{for }i \in \{0,\ldots,N-1\},
\end{align}
}
where $(X_j)_{i\le j \le N}$ is a Markov chain approximating the forward component \eqref{eq:eds} (typically the Euler scheme, see  Algorithm \ref{alg:markov:euler} below), $\Delta W_i:=W_{t_{i+1}} - W_{t_i}$ is the $(i+1)$-th Brownian motion increment, and $\Espi{\cdot}:=\Esp{\cdot\mid \cF_{t_i}}$ is the conditional expectation.
Our working assumptions on the  functions $g$ and $f$ are as follows:
\begin{description}
\item [\HG] $g$ is a bounded measurable function from $\R^d$ to $\R$, the upper bound of which is denoted by $C_g$.
\item [\HF] for every {$i<N$}, $f_i(x,y,z)$ is a measurable function $\R^d\times \R \times \R^q$ to $\R$, and there exist two finite constants $L_f$ and $C_f$ such that, for every $i<N$,
\begin{align}
\label{eq:f:lip}
|f_i(x,y,z)-f_i(x,y',z')|&\leq L_f (|y-y'|+|z-z'|),\\
&\qquad\qquad \forall (x,y,y',z,z')\in \R^d\times(\R)^2 \times (\R^q)^2,\\
\label{eq:f:bound}
|f_i(x,0,0)|&\leq C_f, \quad \forall x\in \R^d.
\end{align}
\end{description}
The definition of the Markov chain $(X_j)_{j}$ is made under the following assumptions.
\begin{description}
\item [\HX]  The coefficients functions $b$ and $\sigma$ satisfy
\begin{romannum}
\item $b:[0,T] \times\R^d \to \R^d$ and $\sigma :[0,T] \times\R^d \to \R^d\otimes \R^q$ are bounded measurable, uniformly Lipschitz in the space dimensions;
\item there exists $\zeta \ge1$ such that, for all $\xi \in \R^d$, the following inequalities hold:
 $\zeta^{-1} |\xi|^2 \le  \xi^\top   \sigma(t,x) \sigma(t,x)^\top \xi \le \zeta |\xi|^2$.
\end{romannum}\end{description}
Let $X_i$ be a random variable with some distribution $\eta$ (more details on this to follow). Then 
$X_j$ for $j> i$ is generated according to one of the two algorithms below:
\begin{algorithm}[SDE dynamics]
\label{alg:markov:SDE}
$X_{j+1}=\bar X_{t_{j+1}}  = X_j + \int_{t_j}^{t_{j+1}} b(s,\bar X_s) \dd s +  \int_{t_j}^{t_{j+1}} \sigma(s,\bar X_s) \dd W_s$;
\end{algorithm}
\begin{algorithm}[Euler dynamics]
\label{alg:markov:euler}
$X _{j+1}  = X_j + b(t_i,X _ i)\Delta_t + \sigma (t_i, X _ i) \Delta W_i$.\\
\end{algorithm}

The above ellipticity condition (ii) will be used in the proof of Proposition \ref{prop:hnu}.

As in the continuous time framework \eqref{eq:fbsde}, the solution of the MDP \eqref{eq:MDP:intro} admits a Markov representation: under \HG, \HF\ and \HX\ (and using for $X$ either the SDE itself or its Euler scheme),
for every $i$, there exist measurable deterministic functions $y_i : \R^d \to \R$ and $z_i:\R^d \to \R^q$, such that
$Y_i = y_i(X_i)$ and $ Z_i = z_i(X_i)$,
 almost surely.
In fact,  the value functions $y_i(\cdot)$ and $z_i(\cdot)$ are independent of how we initialize the forward component.\footnote{Actually under our assumptions, the measurability of $y_i$ and $z_i$ can be easily established by induction on $i$. More precisely, we can write $y_i$ and $z_i$ as a $(N-i)$-fold integrals in space, using the $C^2$ transition density of $X$ given in Algorithms \ref{alg:markov:SDE} or \ref{alg:markov:euler}. From this, we observe that $z_i$ is a $C^2$ function of $x_i$; regarding $y_i$ all the contributions in the sum for $j>i$ are also smooth, and only the $j=i$ term may be non-smooth (because of $x_i\mapsto f_i(x_i,.)$ is only assumed measurable). From this, we easily see that the initialization $x_i$ of $X$ at time $i$ can be made arbitrary, provided that this is independent of $W$.}

For the subsequent stratification algorithm, $X_i$ will be sampled randomly (and independently of the Brownian motion $W$) according to different squared-integrable distributions $\eta$. When $X_i\sim \eta$, we will write $(\X{i,\eta}j)_{i\le j \le N}$ the Markov chain given in \HX, using either the SDE dynamics (better when possible) or the Euler one. One can recover the value functions from the conditional expectations: almost surely,
{\small 
\begin{align}
\label{eq:MDP:fcs}
y_i(\X{i,\eta} i)=\Esp{  g(\X {i,\eta} N) +\sum_{j=i}^{N-1}f_j(\X {i,\eta} j,y_{j+1}(\X{i,\eta}{j+1}),z_j(\X{i,\eta} j)) \Dt \ \big|\ \X{i,\eta}i},\\
\nonumber \Dt z_i (\X {i,\eta}i)   = \Esp{  (g(\X{i,\eta} N) +\sum_{j=i+1}^{N-1} f_j(\X{i,\eta} j,y_{j+1}(\X{i,\eta}{j+1}),z_j(\X{i,\eta} j)) \Dt ) \Delta W_i  \ \big|\ \X{i,\eta}i };
\end{align}
}
the proof of this is the same as \cite[Lemma 4.1]{gobe:turk:14}.

 Approximating the solution to \eqref{eq:MDP:intro} is actually  achieved by approximating the functions $y_i(\cdot)$ and $z_i(\cdot)$.
In this way, we are directly approximating the solution to the semilinear PDE \eqref{eq:pde}.
Our approach consists in approximating the restrictions of the functions $y_i$ and $z_i$ to subsets of a cubic partition of $\R^d$ using finite dimensional linear function spaces.
The basic assumptions for this local approximation approach are given below.
\begin{description}
\item [\HS] There are $K\in \N^*$ disjoint hypercubes $(\cH_k:1\leq k \leq K)$,
that is
\begin{equation}
\label{eq:hs}\cH_k\cap \cH_l=\emptyset,  \qquad\bigcup_{k=1}^K \cH_k=\R^d \quad \text{ and }\quad\cH_k=\prod_{l=1}^d [\xkm{l},\xkp{l})
\end{equation}
for some $-\infty\leq \xkm{l}< \xkp{l}\leq +\infty$.
Additionally, there are linear function spaces $\LinSpace{Y,k}$ and $\LinSpace{Z,k}$, valued in $\R$ and $\R^q$ respectively,
which are subspaces of $\L_2(\cH_k,\nu_k)$ w.r.t. a probability measure $\nu_k$ on $\cH_k$ defined in \Hnu\ below.
\end{description}
Common examples of hypercubes are:
\begin{romannum}
\item Hypercubes of equal size: $\xkp{l}-\xkm{l}={const}>0$ for all $k$ and $l$, except for exterior strata that must be infinite.
\item Hypercubes of equal probability: $\nu(\cH_k)=1/K$ for some probability $\nu$ to be defined later in \Hnu.
\end{romannum}

Common examples of local approximations spaces $\LinSpace{Y,k}$ and $ \LinSpace{Z,k}$ are:
\begin{romannum}
\item Piece-wise constant approximation (\LPz): $\LinSpace{Y,k} := \text{span}\{\1_{\cH_k}\}$, and $  \LinSpace{Z,k} := ( \LinSpace{Y,k})^q$; $\DimLinSpace{Y}=1$ and $\DimLinSpace{Z,k} = q$.
\item Affine approximations (\LPo): $\LinSpace{Y,k} := \text{span} \{\1_{\cH_k},x_{1}\1_{\cH_k}, \dots, x_{d}\1_{\cH_k} \}$, and $  \LinSpace{Z,k}: = ( \LinSpace{Y,k})^q$;
 $\DimLinSpace{Y}=d+1$ and $\DimLinSpace{Z,k} = q(d+1)$.
\end{romannum}

The key idea in this paper is to select a distribution $\nu$, the restriction of which to the hypercubes $\cH_k$, $\nu_k$, can be explicitly computed.
Then, we can easily simulate i.i.d. copies of $\X{i,\nu_k}i$ directly in $\cH_k$ and use the resulting paths of the Markov chain to estimate $y_k(\cdot)|_{\cH_k}$.
This sampling method is traditionally known as {\em stratification}, and for this reason we will call the hypercubes in  \HS \ the {\em strata}.
For the stratification, the components $\X{i,\nu_k}i$ are sampled as i.i.d. conditional logistic random variables, which is precisely stated in the following assumption.

\begin{description}
\item [\Hnu]
Let $\mu >0$.
The distribution of $\X{i,\nu_k}i$ is given by $\P\circ (\X{i,\nu_k}i)^{-1} (\dx) = \nu_k(\dx)$, where
\begin{equation}
\label{eq:nuk}\nu_k(\dx)=\frac{\1_{\cH_k}(x)\nu(\dx)  }{\nu(\cH_k)},
\end{equation}
and
\begin{equation}
\nu(\dx) = \plog(x) \dd x, \quad
\plog(x):=\prod_{l=1}^d \frac{\mu e^{-\mu x_l} }{ (1+e^{-\mu x_l})^2}, \quad x=(x_1,\dots,x_d)\in \R^d.
\label{eq:logistic}
\end{equation}
\end{description}

\begin{remark}
\label{rem:strat}
An important relation of $\nu$ and $\nu_k$ is that one has the $\L_2$-norm identity $\normnu \cdot ^2 = \sum_{k=1}^K \nu(\cH_k)\normnuk \cdot ^2$.
\end{remark}

In order to generate the random variable $\X {i,\nu_k}i$, we make use of the {\em inverse} conditional distribution function of $\nu_k$ and the simulation of uniform random variables, as shown in the following algorithm:

\begin{algorithm}
\label{alg:stratify}
Draw $d$ independent random variables $(U_1,\ldots,U_d)$ which are uniformly distributed on $[0,1]$, and compute
$$\X {i,\nu_k}i:=\left(F^{-1}_{\nu, [\xkm1,\xkp1)}(U_1),\dots,F^{-1}_{\nu, [\xkm{d},\xkp{d})}(U_d)\right)\overset{d}{\sim} \nu_k,$$
where we use the functions $F_\nu(x) := \int_{-\infty}^x \nu(\dx') =  1 / \left(1 +\exp(-\mu x)\right) $ and
$$F^{-1}_{\nu, [x^-,x^+)}(U)=
-\frac{ 1 }{\mu }\log\left(\frac{ 1 }{ F_\nu(x^-)+U(F_\nu(x^+)-F_\nu(x^-))}-1\right).$$
\end{algorithm}

A further reason for the choice of the logistic distribution is that it induces the following stability property on the $\L_2$ norms of the Markov chain $(\X {i,\nu}j)_{i\le j\le N}$;
this property will be crucial for the error analysis of the stratified regression scheme in \S \ref{section:error}.
The proof is postponed to Appendix \ref{proof:proposition:prop:hnu}.
\begin{proposition}\label{prop:hnu}
Suppose that $\nu$ is the logistic distribution defined in \Hnu. %, and suppose that $\X{i,\nu}{}$ is the Markov chain satisfying \HX.
There is a constant $\cnu\in [1,+\infty)$  such that, for any function $h:\R^d\mapsto \R$ or $\R^q$ in $\L_2(\nu)$, for any $0\leq i\leq  N$, and for any $i\le j \le N-1$, we have
\begin{equation}
\label{eq:hnu}
\frac{1  }{\cnu  } \E[|h(\X{i,\nu}j)|^2] \leq \normnu{h}^2\leq \cnu \E[|h(\X{i,\nu}j)|^2].
\end{equation}
\end{proposition}

To conclude this section, we recall standard uniform absolute bounds for the functions $y_i(\cdot)$ and $z_i(\cdot)$.
\begin{proposition}[\emph{a.s}. upper bounds, {\cite[Proposition 3.3]{gobe:turk:14}}]
\label{prop:bound}For $N$ large enough such that $\frac{T }{ N} L^2_f\leq \frac{ 1 }{12 q }$, we have for any $x\in \R^d$ and any $0\leq i \leq N-1$,
\begin{align}
\label{eq:prop:bound}
\quad |y_{i}(x)|    \le C_y :=
e^{\frac T4 +  {6 q (1\lor L^2_f)  } (T\lor 1)}
\Big(C_g+\frac{T}{ 2\sqrt{q}}C_f\Big),  \qquad
 |z_{l,i}(x)| \leq   C_{z} := \frac{C_y}{\sqrt{\Dt} }.
\end{align}
\end{proposition}

%%%%%%%%%%%%%%%%%%%%%%%%%%%%%%%%%%%%
%%%%%%%%%%%%%%%%%%%%%%%%%%%%%%%%%%%%
\section{Stratified algorithm and convergence results}

\subsection{Algorithm}
\label{subsection:algorithm}
In this section, we  define the SRMDP algorithm mathematically, and then expose in \S \ref{section:GPU} how to efficiently perform it using GPUs.
Our algorithm involves solving a  sequence of  Ordinary linear Least Squares regression (\OLS) problems. For a precise mathematical statement, we recall the seemingly abstract but very convenient definition from \cite{gobe:turk:14}; explicit algorithms for the computation of \OLS \ solutions are exposed in \S \ref{sec:explicit coefs}.

\begin{definition}[Ordinary linear least-squares regression]\label{def:ls}
For $l,l'\geq 1$ and for probability spaces $(\tilde \Omega,\tilde \cF,\tilde \P)$ and $(\R^l,\cB(\R^l),\eta)$,
let $S$ be a $\tilde\cF\otimes \cB(\R^l)$-measurable
$\R^{l'}$-valued function such that  $S(\omega,\cdot) \in \L_2(\cB(\R^l),\eta)$ for $\tilde \P$-a.e. $\omega \in \tilde \Omega$,
and $\LinSpace{}$  a linear vector subspace of $\L_2(\cB(\R^l),\eta)$ spanned by deterministic $\R^{l'}$-valued functions $\{p_k(.), k\geq 1\}$.
The least squares approximation of $S$ in the space $\LinSpace{}$ with respect to $\eta$ is the ($\tilde \P\times\eta$-a.e.) unique,
$\tilde \cF\otimes\cB(\R^l)$-measurable function $S^\star$ given by
\begin{equation}
\label{eq:mc:mls}
 S^\star(\omega, \cdot) =\arg\inf_{\phi \in \LinSpace{} }\int |\phi(x) - S(\omega,x) |^2 \eta(\dx).
\end{equation}
We say that $S^\star$ solves $\OLS(S,\LinSpace{},\eta)$.

{On the other hand, suppose that $\eta_M=\frac 1M \sum_{m=1}^M \delta_{\cX^{(m)}}$ is a discrete probability measure
 on  $(\R^l,\cB(\R^l))$,
where $\delta_x$ is the Dirac measure on $x$ and $\cX^{(1)},\ldots,\cX^{(M)}: \tilde \Omega \rightarrow \R^l$ are   i.i.d. random variables.
For an $\tilde\cF\otimes \cB(\R^l)$-measurable $\R^{l'}$-valued function $S$ such that $\big|S\big(\omega,\cX^{(m)}(\omega)\big)\big| < \infty$ for any $m$ and $\tilde \P$-a.e. $\omega \in \tilde\Omega$,
the least squares approximation of $S$ in the space $\LinSpace{}$ with respect to $\eta_M$ is the ($\tilde \P$-a.e.) unique,
$\tilde \cF\otimes\cB(\R^l)$--measurable function $S^\star$ given by
\begin{equation}
\label{eq:mc:ls:empi}
 S^\star(\omega, \cdot) = \arg\inf_{\phi \in \LinSpace{} }
\frac 1M \sum_{m=1}^M |\phi\big(\cX^{(m)}(\omega) \big) - S \big(\omega,\cX^{(m)}(\omega) \big) |^2 .
\end{equation}
We say that $S^\star$ solves $\OLS(S,\LinSpace{},\eta_M)$.
}
\end{definition}

\begin{definition}[Simulations and empirical measures]
\label{def:sims and empi}
Recall the Markov chain $(\X{i,\nu_k}{j})_{i\le j \le N}$ initialized as in  \Hnu.
For any $i\in\{0,\ldots,N-1\}$ and $k\in \{1,\dots,K\}$, define $M\geq { \DimLinSpace{Y,k} \vee \DimLinSpace{Z,k} }$ independent copies of $(\Delta W_i,(X^{i,\nu_k}_j)_{i\leq j \leq N})$ that we denote by
\begin{equation}
\label{eq:cloud:ik}
\cC_{i,k} := \left\{( \DW{i,k,m}{i},(\X{i,k,m}{j})_{i\leq j \leq N} ) \; : \  m = 1,\dots,M \right\}.
\end{equation}
The random variables $\cC_{i,k}$ form a \emph{cloud of simulations} used for the regression at time $i$ and in the stratum $k$.
Furthermore, we assume that the  clouds of simulations $(\cC_{i,k}:0\leq i \leq N-1, 1\leq k \leq K)$ are independently generated. All these random variables are defined on a probability space  $(\Om M, \F M{} , \P^{(M)})$.
Denote by $\nu_{i,k,M}$ the empirical probability measure of the $\cC_{i,k}$-simulations,
i.e.
\begin{equation*}
\nu_{i,k,M} = \frac 1{M} \sum_{m=1}^{M} \delta_{(\DW{i,k,m}i, \X {i,k,m}{i},\ldots,\X{i,k,m}N )}.
\end{equation*}
Denoting by  $(\Omega, \cF, \P)$ the probability space supporting $(\Delta W_i, X^{i,\nu_k}:0\leq i \leq N-1,1\leq k \leq K)$, which serves as a generic element for the  clouds of simulations $\cC_{i,k}$, the full probability space used to analyze our algorithm is the product space $(\bar \Omega, \bar \cF, \bar \P)=(\Omega, \cF,\P) \otimes (\Om M, \F M{} , \P^{(M)})$. By a slight abuse of notation, we write $\P$ (resp. $\E$) to mean $\bar \P$ (resp. $\bar \E$) from now on.
\end{definition}
%An important feature is that simulating $(X^{i,\nu_k}_j:i\leq j \leq N)$ is easy (see Subsection \ref{subsection:SimulationOfXinuk}).

%To avoid over-fitting issues, we assume -- without loss of generality --
%$M\geq \DimLinSpace{Y,k} \vee \DimLinSpace{Z,k}$ for all $k$.

We now come to the definition of the stratified LSMDP algorithm, which computes random approximations $\yM {i}{.}$ and $\zM {i}{.}$

\begin{algorithm}[SRMDP]
\label{alg:srmdp}
Recall  the linear spaces $\LinSpace{Y,k}$ and  $\LinSpace{Z,k}$
from \HS,
 the bounds \eqref{eq:prop:bound} and the truncation function $\cT_L$ (see \eqref{eq:TL}).
\begin{description}
\item [\bf Initialization.]
Set $\yM{N}{\cdot} := g(\cdot)$.

\item [\bf Backward iteration for $i=N-1$ to $i=0$.]
For any stratum index $k\in \{1,\dots, K\}$,
generate the empirical measure $\nu_{i,k,M} $ as in
Definition \ref{def:sims and empi},
and define
%\emma{in the first OLS on $y$, the variables are $\vecx{x}_i$ and the measure is related to higher dimensional variable $(\DW{i,k,m}i, \X {i,k,m}{i},\ldots,\X{i,k,m}N )$, but I think this is clear like this}
\begin{equation}
\label{eq:PsiM}
\left\{\begin{array}{l}
 \displaystyle \psim_{Z,i,k}(\cdot) \quad \text{solution of} \quad
\OLS(\ObsM{Z,i}{w,\vecx{x}_i} \ , \ \LinSpace{Z,k} \ , \ \nu_{i,k,M})\\[2mm]
\displaystyle \qquad \text{for} \quad \ObsM{Z,i}{w,\vecx{x}_i} :=
\frac 1{\Dt}\ObsM{Y,i+1}{\vecx{x}_i} \ w,\\[3mm]
\zM i\cdot|_{\cH_k} := \cT_{{C_{z}}}\big( \psim_{Z,i,k}(\cdot) \big)\quad (truncation),\\[4mm]
\displaystyle \psim_{Y,i,k}(\cdot)  \quad \text{solution of} \quad
\OLS(\ObsM {Y,i}{\vecx{x}_i} \ ,  \ \LinSpace{Y,k} \ , \ \nu_{i,k,M})
 \\
\displaystyle \qquad \text{for} \quad \ObsM{Y,i}{\vecx{x}_i} := g(x_N) + \sum_{j=i}^{N-1} f_j\big(x_j, \yM {j+1}{x_{j+1}}, \zM{j}{x_j} \big ) \Dt,
\\[2mm]
\yM i\cdot|_{\cH_k} := \cT_{C_y}\big(\psim_{Y,i,k}(\cdot) \big)\quad (truncation),
\end{array}\right.
\end{equation}
where $w \in \R^q$ and $\vecx{x}_i = (x_i,\dots,x_N) \in (\R^d)^{N-i+1}$.

%\item[Truncation.]
%For every $k=1,\dots,K$, define
%\begin{romannum}
%\item $\yM i\cdot|_{\cH_k} := \cT_{C_y}\big(\psim_{Y,i,k}(\cdot) \big)$,
%\item $\zM i\cdot|_{\cH_k} = \cT_{{C_{z}}}\big( \psim_{Z,i,k}(\cdot) \big)$.
%\end{romannum}
\end{description}
\end{algorithm}
An important difference between SRMDP and established Monte Carlo algorithms  \cite{gobe:lemo:wari:05, lemo:gobe:wari:06,gobe:turk:13b,gobe:turk:14} is that the number of simulations falling in each hypercube is no
 more random but fixed and equal to $M$.
Observe first that this is likely to improve the numerical stability of the regression algorithm:
there is no risk that too few simulations will land in the hypercube, leading to under-fitting.
Later, in \S \ref{section:GPU}, we shall explain how to implement Algorithm \ref{alg:srmdp} on a GPU device.
The key point is that the calculations at every time point are fully independent between the different hypercubes, so that we can perform them in parallel across the hypercubes.
The choice of $M$ independent on $k$ is made in order to maintain a computational effort equal on each of the strata.
In this way,  the gain in parallelization is likely  to be the largest.
However, the subsequent mathematical analysis can be easily adapted to make the number of simulations vary with $k$ whenever necessary.

An easy but important consequence of Algorithm \ref{alg:srmdp} and of the bounds of Proposition \ref{prop:bound} is the following absolute bound; the proof is analogous to that of \cite[Lemma 4.7]{gobe:turk:14}. % and is left to the reader.
\begin{lemma}\label{cor: as 2} With the above notation, we have
\[ \sup_{0\leq i \leq N}\sup_{\vecx{x_i}\in(\R^d)^{N-i+1}}|\ObsM{Y,i}{\vecx{x_i}}| \leq \Cas := C_g + T\left(L_f C_y \left[1 + \frac{ \sqrt q }{\sqrt {\Dt} }
\right]+C_f \right).
\]
\end{lemma}

\subsection{Error analysis}
\label{section:error}
The analysis will be performed according to several $\L_2$-norms, either w.r.t. the probability measure $\nu$, or the empirical norm related to the cloud simulations. They are defined as follows:
\begin{align}
\label{eq:cE(Y,M,i)}
\cE(Y,M,i) &:= \sumk\nu(\cH_k)\Esp{\empinorm{\yM {i} \cdot - y_i(\cdot)}{i,k}^2} , \\
\label{eq:bar:cE(Y,M,i)}
\bar \cE(Y,M,i) &:= \sumk\nu(\cH_k)\Esp{\normnuk{\yM {i} \cdot - y_i(\cdot)}^2} =\Esp{\normnu{\yM {i} \cdot - y_i(\cdot)}^2}, \\
\label{eq:cE(Z,M,i)} \cE(Z,M,i)& :=\sumk\nu(\cH_k) \Esp{ \empinorm{\zM {i} \cdot - z_i(\cdot)}{i,k}^2},\\
\label{eq:bar:cE(Z,M,i)} \bar\cE(Z,M,i)& :=\sumk\nu(\cH_k) \Esp{ \normnuk{\zM {i} \cdot - z_i(\cdot)}^2}=\Esp{\normnu{\zM {i} \cdot - z_i(\cdot)}^2} \, ,
  \end{align}
where
\begin{equation}
\label{eq:def:empinorm}
\empinorm{h}{i,k}:=\left(\int |h|^2(\omega,\vecx{x}_i)\nu_{i,k,M}(\dd\omega,\dd\vecx{x}_i)\right)^{1/2}.
\end{equation}
In fact, the norms  $\cE(.,M,i)$ and $\bar \cE(.,M,i)$ are related through model-free concentration-of-measures inequalities.
This relation is summarized in the proposition below.
\begin{proposition}\label{prop:eq:y:z:err:deco:M}
For each $i \in \{0,\dots , N-1\}$,  we have
\begin{align}
  &\bar \cE(Y,M,i)
 \le 2 \cE(Y,M,i)
+  \frac{2028 C_y^2 \log(3M)}{M} \left(\nuDimLinSpace{Y,.}+1\right),
\nonumber
\\
&\bar \cE(Z,M,i) \le 2\cE(Z,M,i)
+  \frac{2028   qC_y^2 \log(3M) }{\Dt M}\left(\nuDimLinSpace{Z,.}+1\right).
\nonumber
\end{align}
\end{proposition}
\begin{proof}
It is clearly sufficient to show that
\begin{align*}
\Esp{\normnuk{\yM {i} \cdot - y_i(\cdot)}^2} & \le 2 \Esp{\empinorm{\yM {i} \cdot - y_i(\cdot)}{i,k}^2} \\
& \qquad + \frac{2028 C_y^2 \log(3M)}{M} \left(\DimLinSpace{Y,.}+1\right) ,\\
\Esp{\normnuk{\zM {i} \cdot - z_i(\cdot)}^2} & \le 2 \Esp{\empinorm{\zM {i} \cdot - z_i(\cdot)}{i,k}^2} \\
&\qquad +  \frac{2028   qC_y^2 \log(3M) }{\Dt M}\left(\DimLinSpace{Z,.}+1\right),
\end{align*}
which follows exactly as in the proof of \cite[Proposition 4.10]{gobe:turk:14}.
\end{proof}

From the previous proposition, the controls on $\bar \cE(Y,M,i)$ and $\bar \cE(Z,M,i)$ stem from those on $\cE(Y,M,i)$ and $\cE(Z,M,i)$, which are handled in Theorem \ref{thm:mc:err decomp} below.\\
In order to study the impact of basis selection, we define the squared  {\em quadratic approximation errors} associated to the basis in hypercube $\cH_k$ by
\[
T^{Y}_{i,k} := \inf_{\phi \in \LinSpace{Y,k}}\normnuk{\phi - y_i}^2, \quad T^{Z}_{i,k} := \inf_{\phi \in \LinSpace{Z,k}}\normnuk{\phi - z_i}^2.
\]
These terms are the minimal error that can possibly be achieved by the basis $\LinSpace{Y,k}$ (resp. $\LinSpace{Z,k}$) in order to approximate the restriction $y_i(\cdot)|_{\cH_k}$ (resp. $z_i(\cdot)|_{\cH_k}$) in the $\L_2$ norm. Consequently, the global squared quadratic approximation error is given by
\begin{align}
\label{eq:TY:1i}
T^{Y}_{i}  := \sumk \nu(\cH_k)T^Y_{i,k} =\inf_{\phi\text{ s.t. }\phi|_{\cH_k} \in \LinSpace{Y,k}}\normnu{\phi - y_i}^2,\\
\label{eq:TZ:1i}
T^{Z}_{i}  :=\sumk \nu(\cH_k)T^Z_{i,k}=\inf_{\phi\text{ s.t. }\phi|_{\cH_k} \in \LinSpace{Z,k}}\normnu{\phi - z_i}^2.
\end{align}
As we shall see in Theorem \ref{thm:mc:err decomp} below, the terms $T^Y_i$ and $T^Z_i$ are closely associated to the limit of the expected quadratic error of the numerical scheme in the asymptotic $M\to \infty$; for this reason, these terms are usually called {\em bias} terms.

%which are the best quadratic errors using the basis functions. The second inequalities in \eqref{eq:TY:1i} and \eqref{eq:TZ:1i} are due to the definitions of $\nu_k$ and $\nu$, and of the strata $(\cH_k:1\leq k \leq K)$ partitioning $\R^d$.

Now, we are in the position to state our main result giving non-asymptotic error estimates.\begin{theorem}[{Error for the Stratified LSMDP scheme}]
\label{thm:mc:err decomp}
Recall the constants $C_y$ from Proposition \ref{prop:bound}, $\Cas$ from Lemma \ref{cor: as 2}, and $\cnu$ from Proposition \ref{prop:hnu}.
For each $i \in \{0,\dots,N-1\}$, define
\begin{align*}
\cE(i)& := 2\sumji \Dt \Big(T^{Y}_{j}+3\Cas^2 \frac{\nuDimLinSpace{Y,.}  }{M }
+12168 L^2_f \Dt \frac{
 (\nuDimLinSpace{Z,.}+1)  qC_y^2 \log(3 M) }{  M}\\
&\hspace{2cm}+3T^{Z}_{j}+6q\Cas^2\frac{ \nuDimLinSpace{Z,.}}{\Dt  M}\Big)\\ &\hspace{5mm}+
(T-t_i)\frac{1014 C_y^2 \log(3M)}{M} \left(\left(\nuDimLinSpace{Y,.}+1\right)+\frac{ q }{\Dt }\left(\nuDimLinSpace{Z,.}+1\right)\right).
\end{align*}
For  $\Delta_t$ small enough such that $L_f \Dt\leq \sqrt{\frac{2 }{ 15}}$ and $\Delta_t L^2_f\leq \frac1{288 \cnu^2 \cprop (1+T)}$, we have, for all
$0\leq i \leq N-1$,
\begin{align}
\cE(Y,M,i) &\le  T^{Y}_{i}
 +3\Cas^2 \frac{\nuDimLinSpace{Y,.}  }{M }+ 12168 L^2_f \Dt \frac{
 (\nuDimLinSpace{Z,.}+1)  qC_y^2 \log(3 M) }{  M} \nonumber\\
&\quad +(1+15L_f^2 \Dt )\cgamma\cE(i),\label{eq:empi:y} \\
& \hspace{-1cm}\sumji  \Dt\cE(Z,M,j) \le
\cgamma \cE(i),
\label{eq:empi:z}
\end{align}
where $\cgamma := \exp(288 \cnu^2 \cprop (1+T)L^2_f T)$.
\end{theorem}

 \subsection{Proof of Theorem \ref{thm:mc:err decomp}}
\label{subsection:proof:main:thm}
We start by obtaining estimates on the {\em local} empirical quadratic errors terms 
\[
\Esp{\empinorm{\yM {i} \cdot - y_i(\cdot)}{i,k}^2}, \quad \Esp{\empinorm{\zM {i} \cdot - z_i(\cdot)}{i,k}^2},
\]
on each of the hypercubes $\cH_k$ ($k=1,\ldots,K$).
We first reformulate \eqref{eq:MDP:fcs} with $\eta=\nu_k$ in terms of the Definition \ref{def:ls} of OLS.
For each $i\in \{0,\dots,N-1\}$ and $k\in \{1,\dots,K\}$, let $\nu_{i,k} := \P \circ (\Delta W_i,X^{i,\nu_k}_i ,\dots,X^{i,\nu_k}_N)^{-1}$, so that
we have
\begin{align}
\left\{
\begin{array}{l}
 y_i(\cdot)|_{\cH_k}  \text{ solution of }
\OLS( \ \Obs{Y,i}{\vecx{x}_i} \ , \ \LinSpace{ k}^{(1)}
\ , \ \nu_{i,k} \ ) \\
\hspace{8mm} \displaystyle \ \text{where} \ \Obs{Y,i}{\vecx{x}_i}  := g(x_N) + \sum_{j=i}^{N-1} f_j\big(x_j,y_{j+1}(x_{j+1}), z_j(x_j)\big)\Dt,\\
 z_i(\cdot)|_{\cH_k} \text{ solution of }
\OLS( \ \Obs{Z,i}{w,\vecx{x}_i} \ , \  \LinSpace{ k}^{(q)}
\ , \ \nu_{i,k} \ )\\[2mm]
\hspace{8mm} \displaystyle \ \text{where} \ \Obs{Z,i}{w,\vecx{x}_i} :=  \frac 1{\Dt}\Obs{Y,i+1}{\vecx{x}_i}\ w,
\end{array}
\right.
\label{eq:MDP:ls:def}
\end{align}
where $w \in \R^q$, $\vecx{x}_i := (x_i,\dots,x_N) \in (\R^d)^{N-i+1}$ and where   $\LinSpace{k}^{(l')}$ is any dense separable subspace  in the $\R^{l'}$-valued functions belonging to $\L_2(\cB(\cH_k) ,  \nu_k )$.
The above OLS solutions and those defined in \eqref{eq:PsiM} will be compared with other intermediate OLS solutions given by
\begin{equation}
\label{eq:intermediate:empirical}
\left\{\begin{array}{l}
\displaystyle \PsiM{Y,i,k}{\cdot}  \quad \text{solution of} \quad
\OLS(\ \Obs {Y,i}{\vecx{x}_i} \ ,  \ \LinSpace{Y,k} \ , \ \nu_{i,k,M}),
 \\[2mm]
 \displaystyle \PsiM{Z,i,k}{\cdot} \quad \text{solution of} \quad
\OLS( \ \Obs{Z,i}{w,\vecx{x}_i} \ , \ \LinSpace{Z,k} \ , \ \nu_{i,k,M}).
\end{array}\right.
\end{equation}

In order to handle the dependence on the simulation clouds, we define the following $\sigma$-algebras.
\begin{definition}
\label{def:sim cnd}
Define the $\sigma$-algebras
\begin{align*}
\F *i   :=   \sigma(\cC_{i+1,k},\dots,{\cC_{N-1,k}}:1\leq k \leq K),\quad
\F {M}{i,k}  :=  \F *i \vee \sigma(\X{i,k,m}{i} : \; 1 \le m \le M).
\end{align*}
For every $i\in\{0,\ldots,N-1\}$ and $k\in\{1,\dots,K\}$, let
$\EspMik{\cdot}$ (resp. $\P^M_{i,k}\left(\cdot\right)$) with respect to $\F M{i,k}$.
\end{definition}
Defining additionally the functions
\begin{align*}
%\txiM{Y, i}x &:=\Esp{\tObsM{Y,i}{\vecx{X}_i}-\Obs{Y,i}{\vecx{X}_i}\mid X_i=x,\F M{}}, \\
\xiM{Y,i}x &:= \Esp{\ObsM{Y,i}{\vecx{X}_i}-\Obs{Y,i}{\vecx{X}_i}\mid X_i=x,\F M{}},\\
\xiM{Z,i}x&:=\Esp{\ObsM{Z,i}{\Delta W_i,\vecx{X}_i}-\Obs{Z,i}{\Delta W_i,\vecx{X}_i}\mid X_i=x,\F M{}},
\end{align*}
now we are in the position to prove that
\begin{align}
\Esp{\empinorm{ y_i(\cdot) - \yM{i}{\cdot} } {i,k}^2}
&\le   T^{Y}_{i,k}
 + 6 \Esp{\normnuk{  \xiM{Y,i}{\cdot}} ^2} +3\Cas^2 \frac{\DimLinSpace{Y,k}  }{M } \, ,\\
&\hspace{0cm}  \quad +15L_f^2 \Dt^2 \Esp{\empinorm{ z_i(\cdot) - \zM i\cdot}{i,k} ^2}\\
& + 12168 L^2_f \Dt \frac{
 (\DimLinSpace{Z,k}+1)  qC_y^2 \log(3 M) }{  M}
\label{eq:y:empi:decom},\\
\Esp{\empinorm{ z_i(\cdot) - \zM{i}{\cdot} } {i,k}^2}
& \le  T^{Z}_{i,k}
 + 2 \Esp{\normnuk{\xiM{Z,i}{\cdot}}^2}
 +2q\Cas^2\frac{ \DimLinSpace{Z,k}}{\Dt  M}.
 \label{eq:z:empi:decom}
\end{align}
In fact, the proof of \eqref{eq:y:empi:decom}--\eqref{eq:z:empi:decom} follows analogously the proof of \cite[(4.12)--(4.13)]{gobe:turk:14};
in order to follow the steps of that proof, one must note that the term $R_\pi$ of that paper is equal to $1$ here, $C_\pi$ is equal to $\Dt$, and $\theta_L=1$.
Moreover, one must exchange all norms, \OLS \ problems, $\sigma$-algebras, and empirical functions from the reference to the localized versions defined in the preceding paragraphs.
Indeed, the proof method of \cite[(4.12)--(4.13)]{gobe:turk:14} is model free in the sense that it does not care about the distribution of the Markov chain at time $t_i$.

We now aim at aggregating the previous estimates across the strata and  propagating them along time. For this, let
\begin{align}
\nonumber\cE_1(i) := \sumji \Dt \Big( &T^{Y}_{j}+3\Cas^2 \frac{\nuDimLinSpace{Y,.}  }{M }
%\\&\quad
+12168 L^2_f \Dt \frac{
 (\nuDimLinSpace{Z,.}+1)  qC_y^2 \log(3 M) }{  M}\\
&\qquad +3T^{Z}_{j}+6q\Cas^2\frac{ \nuDimLinSpace{Z,.}}{\Dt  M}\Big) \Gamma_j,
\label{eq:E1}
\end{align}
where $\Gamma_i :=(1+\gamma \Dt)^i$ with $\gamma$ to be determined below.
%The stability inequalities on $\nu$ (see \eqref{eq:hnu}) come into play here, as well as stability estimates for discrete BSDEs stated in Proposition \ref{prop: apriori} (with a good choice of parameter $\gamma$).
Next, defining
\begin{equation}
\label{eq:choice:gamma}
\gamma := 288 \cnu^2 \cprop (1+T)L^2_f.
\end{equation}
and recalling that $\Delta_t L^2_f\leq \frac1{288 \cnu^2 \cprop (1+T)}$, then $\gamma$ and $\Dt$ satisfy
\begin{align}
\label{eq:constraint:1:gamma}\qquad 
\max\left(\frac{ 1 }{\gamma } \times 12  \cnu^2 \cprop (1+T)L^2_f ,
\Dt\times 12  \cnu^2 \cprop (1+T) L^2_f \right) \leq \frac{ 1 }{6 }\times\frac{ 1 }{4 }.
\end{align}
%\begin{align}
%\label{eq:constraint:1:gamma}
%\frac{ 1 }{\gamma } \times 12  \cnu^2 \cprop (1+T)L^2_f & \leq \frac{ 1 }{6 }\times\frac{ 1 }{4 },\\
%\label{eq:constraint:2:delta}
%\Dt\times 12  \cnu^2 \cprop (1+T) L^2_f & \leq \frac{ 1 }{6 }\times\frac{ 1 }{4 }.
%\end{align}
Additionally, $\Gamma_i \leq \exp(\gamma T):=\cgamma$ for every $0\leq i\leq N$.
Now, multiply \eqref{eq:y:empi:decom} and \eqref{eq:z:empi:decom} by $\nu(\cH_k)\Dt \Gamma_i$ and sum them up over $i$ and $k$ to ascertain that
{\small
\begin{align}
\label{eq:aggregation:0}
\sumji \Dt\cE(Y,M,j)  &\Gamma_j +\sumji \Dt\cE(Z,M,j)  \Gamma_j \\
\leq\sumji \Dt &\left(T^{Y}_{j}+3\Cas^2 \frac{\nuDimLinSpace{Y,k}  }{M }+12168 L^2_f \Dt \frac{
 (\nuDimLinSpace{Z,k}+1)  qC_y^2 \log(3 M) }{  M}\right) \Gamma_j\\
&\hspace{-23mm}+\sumji \Dt \left\{ \left(T^{Z}_{j}+2q\Cas^2\frac{ \nuDimLinSpace{Z,k}}{\Dt  M} +2\Esp{\normnu{\xiM{Z,j}{\cdot}}^2} \right)(1+15L^2_f \Dt^2) + 6 \Esp{\normnu{  \xiM{Y,j}{\cdot}} ^2} \right\} \Gamma_j \\
&\hspace{-15mm}\leq\cE_1(i)+ 6 \sumji \Dt \left( \Esp{\normnu{  \xiM{Y,j}{\cdot}} ^2}+\Esp{\normnu{\xiM{Z,j}{\cdot}}^2}\right) \Gamma_j \, ,
\label{eq:aggregation:1}
\end{align}
}
where we have used $(1+15L^2_f \Dt^2)\leq 3$ (since $L_f \Dt\leq \sqrt{\frac{2 }{ 15}}$), and the term $\cE_1$ from \eqref{eq:E1} above.
%\begin{align}
%&\sumji \Dt\cE(Y,M,j)  \Gamma_j +\sumji \Dt\cE(Z,M,j)  \Gamma_j
% \\
%&\leq\cE_1(i)+ 6 \sumji \Dt \left( \Esp{\normnu{  \xiM{Y,j}{\cdot}} ^2}+\Esp{\normnu{\xiM{Z,j}{\cdot}}^2}\right) \Gamma_j
%\label{eq:aggregation:1}
%\end{align}
Next, from Proposition \ref{prop:hnu}, we have
\begin{equation}
\label{eq:aggregation:cnu}
\Esp{\normnu{  \xiM{Y,j}{\cdot}} ^2}+\Esp{\normnu{\xiM{Z,j}{\cdot}}^2}
\leq \cnu\left(\Esp{| \xiM{Y,j}{\xnuz_j}| ^2}+
\Esp{| \xiM{Z,j}{\xnuz_j}| ^2}\right).
\end{equation}
Furthermore, note 
 that $(\xiM{Y,j}{\xnuz_j},\xiM{Z,j}{\xnuz_j}:0\leq j \leq N-1)$ solves a discrete BSDE (in the sense of Appendix \ref{section:apriori}) with terminal condition $0$ and driver
\begin{equation}
\label{eq:aggregation:stability:bsde}
f_{\xi^*,j} (y,z) :=  f_j(\xnuz_j,\yM {j+1}{\xnuz_{j+1}}, \zM j{\xnuz_j} )-f_j(\xnuz_j,y_{j+1}(\xnuz_{j+1}), z_j(\xnuz_j)).
\end{equation}
This allows the application of Proposition \ref{prop: apriori}, with the first BSDE $(\xiM{Y,j}{\xnuz_j},\xiM{Z,j}{\xnuz_j}:0\leq j \leq N-1)$, and the second one equal to 0:  since $L_{  f_{2}}=0$, any choice of $\gamma>0$ is valid and we take $\gamma$ as in \eqref{eq:choice:gamma}.
We obtain
\begin{align}
\label{eq:aggregation:2}
& \sumji \Dt \left( \Esp{\normnu{  \xiM{Y,j}{\cdot}} ^2}+\Esp{\normnu{\xiM{Z,j}{\cdot}}^2}\right) \Gamma_j\\
&\leq   6 \cnu \cprop (1+T) \left(\frac{ 1 }{\gamma } + \Dt\right) L^2_f
\sumji   \Dt   \\
&\times \left[\Esp{|\yM {j+1}{\xnuz_{j+1}}-y_{j+1}(\xnuz_{j+1})|^2} +\Esp{|\zM {j}{\xnuz_{j}}-z_{j}(\xnuz_{j})|^2}\right]\Gamma_j.
\end{align}
Now, Proposition \ref{prop:hnu} yields to
\begin{align}
&\Esp{|\yM {j+1}{\xnuz_{j+1}}-y_{j+1}(\xnuz_{j+1})|^2} +\Esp{|\zM {j}{\xnuz_{j}}-z_{j}(\xnuz_{j})|^2} \\
&\le \cnu[\bar \cE(Y,M,j+1)+\bar \cE(Z,M,j)]\\
&\le 2\cnu[\cE(Y,M,j+1)+\cE(Z,M,j)] + \cnu \frac{2028 C_y^2 \log(3M)}{M} \left(\nuDimLinSpace{Y,.}+1\right) \\
& \qquad + \cnu \frac{2028   qC_y^2 \log(3M) }{\Dt M}\left(\nuDimLinSpace{Z,.}+1\right) \, ,
\end{align}
where the last inequality follows from the concentration-measure inequalities in Proposition \ref{prop:eq:y:z:err:deco:M}.
In order to summarize this, we define
\begin{align}
\cE_2(i):= \frac{1014 C_y^2 \log(3M)}{M}  \left(\sumji \Dt\Gamma_j\right) \Bigg(\left(\nuDimLinSpace{Y,.}+1\right)+\frac{q }{\Dt }\left(\nuDimLinSpace{Z,.}+1\right)\Bigg)
\label{eq:ce2i}
\end{align}
and make use of  \eqref{eq:constraint:1:gamma}%and \eqref{eq:constraint:2:delta}
, and that $\Gamma_j \le \Gamma_{j+1}$ in order to ascertain that we have
 \begin{align*}
%\label{eq:aggregation:3}
& \sumji \Dt \left( \Esp{\normnu{  \xiM{Y,j}{\cdot}} ^2}+\Esp{\normnu{\xiM{Z,j}{\cdot}}^2}\right) \Gamma_j\\
%\nonumber
&\leq   12 \cnu^2 \cprop (1+T) \left(\frac{ 1 }{\gamma } + \Dt\right) L^2_f
\left[\sumji   \Dt  \left(\cE(Y,M,j)+\cE(Z,M,j)\right)\Gamma_j+\cE_2(i)\right]\\
%\nonumber
&\leq   \frac{ 1 }{6 } \times\frac{ 1 }{2 }
\left[\sumji   \Dt  \left(\cE(Y,M,j)+\cE(Z,M,j)\right)\Gamma_j+\cE_2(i)\right].
\end{align*}
By plugging this into \eqref{eq:aggregation:1} readily yields to
\begin{align*}
\sumji \Dt\cE(Y,M,j)  \Gamma_j
&+\sumji \Dt\cE(Z,M,j)  \Gamma_j
 \\
&\leq\cE_1(i)+ \frac{ 1 }{2 }
\left[\sumji   \Dt  \left(\cE(Y,M,j)+\cE(Z,M,j)\right)\Gamma_j+\cE_2(i)\right]
\end{align*}
and therefore
\begin{align}
\sumji \Dt\cE(Y,M,j)  \Gamma_j +\sumji \Dt\cE(Z,M,j)  \Gamma_j
\leq2\cE_1(i)+ \cE_2(i).
\label{eq:aggregation:4}
\end{align}
This completes the proof of the estimate \eqref{eq:empi:z} on $z$ as stated in Theorem \ref{thm:mc:err decomp}, using $1\leq \Gamma_i\leq \cgamma$ and $2\cE_1(i)+\cE_2(i)\leq \cgamma \cE(i)$.
It remains to derive  \eqref{eq:empi:y}. 
Starting from \eqref{eq:y:empi:decom}, multiplying  by $\nu(\cH_k)$ and summing over $k$ yields to
\begin{align}
\nonumber\cE(Y,M,i)
\le&   T^{Y}_{i}
 + 6 \Esp{\normnu{  \xiM{Y,i}{\cdot}} ^2} +3\Cas^2 \frac{\nuDimLinSpace{Y,.}  }{M }\\
\nonumber& +15L_f^2 \Dt (2\cE_1(i)+\cE_2(i))\\
& + 12168 L^2_f \Dt \frac{
 (\nuDimLinSpace{Z,.}+1)  qC_y^2 \log(3 M) }{  M}
 \label{eq:aggregation:5}
 \end{align}
 where we use the inequality \eqref{eq:aggregation:4} to control $\Dt\cE(Z,M,i)$. Using the same arguments as before, 
 we upper bound $\Esp{\normnu{  \xiM{Y,i}{\cdot}} ^2}$ by
\begin{align*}
6 \cnu^2 \cprop  \left(\frac{ 1 }{\gamma } + \Dt\right) L^2_f
\sumji   \Dt  \left(\bar\cE(Y,M,j)+\bar\cE(Z,M,j)\right)\Gamma_j.
%\label{eq:y:empi:decom:suite}
\end{align*}
By additionally  bounding $\bar\cE(Y,M,j)$ and $\bar\cE(Z,M,j)$ using the concentration-measure inequalities of Proposition \ref{prop:eq:y:z:err:deco:M} and plugging this in \eqref{eq:aggregation:5}, we finally obtain
 %By using similar arguments as before \obs{ to bound $ \Esp{\normnu{  \xiM{Y,i}{\cdot}} ^2}  +\sumji \Dt  \Esp{\normnu{\xiM{Z,j}{\cdot}}^2} $}, we obtain
%\begin{align*}
%\cE(Y,M,i)
%&\le   T^{Y}_{i}
% + 6 \Esp{\normnu{  \xiM{Y,i}{\cdot}} ^2} +3\Cas^2 \frac{\nuDimLinSpace{Y,k}  }{M }\\
%&\hspace{0cm}  \quad +15L_f^2 \Dt \sumji \Dt\cE(Z,M,j)\\
%& + 12168 L^2_f \Dt \frac{
% (\nuDimLinSpace{Z,k}+1)  qC_y^2 \log(3 M) }{  M} \\
%\cE(Z,M,i)
%& \le  T^{Z}_{i}
% + 2 \Esp{\normnu{\xiM{Z,i}{\cdot}}^2}
% +2q\Cas^2\frac{ \nuDimLinSpace{Z}}{\Dt  M}.
% \end{align*}
%Then, substituting the second above inequality into the first one, we obtain
%\begin{align*}
%\cE(Y,M,i)
%&\le   T^{Y}_{i}
% + 6 \Esp{\normnu{  \xiM{Y,i}{\cdot}} ^2} +3\Cas^2 \frac{\nuDimLinSpace{Y,k}  }{M }\\
%&\hspace{0cm}  \quad +15L_f^2 \Dt \sumji  \Dt\left\{ T^{Z}_{j}
% + 2 \Esp{\normnu{\xiM{Z,j}{\cdot}}^2}
% +2q\Cas^2\frac{ \nuDimLinSpace{Z}}{\Dt  M } \right\} \\
%& + 12168 L^2_f \Dt \frac{
% (\nuDimLinSpace{Z,k}+1)  qC_y^2 \log(3 M) }{  M} \, .
% \end{align*}
% 
%By using similar arguments as before \obs{ to bound $ \Esp{\normnu{  \xiM{Y,i}{\cdot}} ^2}  +\sumji \Dt  \Esp{\normnu{\xiM{Z,j}{\cdot}}^2} $}, we obtain
\begin{align*}
%\nonumber
\cE(Y,M,i)& \le  T^{Y}_{1,i}
 +3\Cas^2 \frac{\nuDimLinSpace{Y,.}  }{M }+15L_f^2 \Dt \left(2\cE_1(i)+ \cE_2(i)\right)\\
&
 + 12168 L^2_f \Dt \frac{
 (\nuDimLinSpace{Z,.}+1)  qC_y^2 \log(3 M) }{  M}\\
& + 72 \cnu^2 \cprop  \left(\frac{ 1 }{\gamma } + \Dt\right) L^2_f
\left[\sumji   \Dt  \left(\cE(Y,M,j)+\cE(Z,M,j)\right)\Gamma_j+\cE_2(i)\right].
%\label{eq:y:empi:decom:suite}
\end{align*}
 From \eqref{eq:constraint:1:gamma} %-\eqref{eq:constraint:2:delta} 
 and \eqref{eq:aggregation:4}, the last term in previous inequality is bounded by
\begin{align*}
\left(\frac{ 1 }{4(1+T) }+\frac{ 1 }{4(1+T) }\right)\left(2\cE_1(i)+\cE_2(i)+\cE_2(i)\right)\leq \cE_1(i)+\cE_2(i)\leq2\cE_1(i)+\cE_2(i).
\end{align*}
This completes the proof of \eqref{eq:y:empi:decom}, using again $2\cE_1(i)+\cE_2(i)\leq \cgamma \cE(i)$.
\qed

\section{GPU implementation}
\label{section:GPU}
In this section, we consider the computation of $\yM i\cdot$  for a given stratum $\cH_k$ and  time point $i$.
The calculation of $\zM i\cdot$ is rather similar, only requiring component-wise calculations  to be taken into account, so that we do not provide details.
The theoretical description of the calculation was given in \S \ref{subsection:algorithm}.
In this section, we first describe the required computations to implement the approximations with \LPz \ and \LPo \ local polynomials in \S \ref{sec:explicit coefs}, and then present their implementation on the GPU in \S \ref{sec:GPU pseudo}.

\footnote{
Theoretically, we are not restricted from going to higher order local polynomials.
We restrict to \LPo \ for  implementation in GPU due to memory limitations. In our forthcoming numerical examples, the GPU's global memory is limited to 6GB. Higher order polynomials would require not only more memory per hypercubes but also more memory for storing the regression coefficients, therefore we would be able to estimate over fewer hypercubes in parallel. 
%Thus, the GPU would be able to run phsysically in parallel less hypercubes. 
Note that this is not an issue for parallel computing on CPUs.
%and therefore  permit fewer threads to be run in parallel. One future improvement, which we cannot implement at present because the GPU CUBLAS library does not yet have this functionality, is to use the free threads in the higher order implementation in order to perform the QR factorization, see \eqref{eq:lp1:explicit} below. This is no an issue for parallel computations on CPUs.
}

\subsection{Explicit solutions to \OLS \ in Algorithm \ref{alg:srmdp}}
\label{sec:explicit coefs}
\paragraph{\LPz}
This piecewise solution is given by the simple formula \cite[Ch. 4]{gyor:kohl:krzy:walk:02}
\begin{equation}
\label{eq:lp0:explicit}
\yM i\cdot |_{\cH_k} = \cT_{C_y}\left( \frac{\sum_{m=1}^{M } \ObsM {Y,i}{\vecX{i,k,m} {i}   }  }{M} \right).
\end{equation}
Observe that there will be a memory consumption of $O(1)$ per hypercube to store the simulations needed for the computation of $\ObsM {Y,i}{\vecX{i,k,m} {i}}$.
Once added in the sum \eqref{eq:lp0:explicit}, their allocation can be freed.

\paragraph{\LPo}
Let $A$ be the $\R^{M} \otimes \R^{d+1}$ matrix, the components of which are given by  $A[m,j] = 1 \1_{\{0\}} (j) + \X{i,k,m}{i,j} \1_{\{0\}^c}(j)$, where $ \X{i,k,m}{i,j}$ is the $j$-th component of $ \vecX{i,k,m}{i}$, and let $S$ be the $\R^{M}$ vector given by $S[m] = \ObsM {Y,i} {\vecX{i,k,m}i}$.
 In order to compute $\yM i\cdot|_{\cH_k}$, we first perform a QR-factorization $A=QR$, where $Q$ is an $\R^{M}\otimes\R^{M}$ orthogonal matrix, and $R$ is an $\R^{M} \otimes \R^{d+1}$ upper triangular matrix. The computational cost to compute this factorization is $ (d+1)^2\left(M - (d+1)/3 \right)$ flops  using the Householder reflections method \cite[Alg. 5.3.2]{golu:vanl:96}. Using the form of \LPo\ and the density of $\nu_k$, we can prove that the rank of $A$ is $d+1$ with probability  1, i.e. $R$ is invertible a.s. (the OLS problem is non-degenerate).

Then, we obtain the approximation $\yM i\cdot|_{\cH_k}$ by computing the coefficients $\alpha = (\alpha_0,\ldots, \alpha_d ) \in\R^{d+1}$ using the QR factorization and backward-substitution method as follows:
\begin{equation}
\label{eq:lp1:explicit}
R \alpha = Q^\top S, \qquad
\yM i{x(k)} = \cT_{C_y} \left( \alpha_0 + \sum_{j=1}^{d} \alpha_{j} \times x_j(k) \right) \, ,
\end{equation}
for any vector $\left(x(k) = (x_1(k),\ldots,x_d(k)\right)$ in $\cH_k$. By using the Householder reflection algorithm for computing the QR-factorization, there will be memory consumption of $O\left(M\times(d+1) \right)$ for the storage of the matrix $A$  on each hypercube.
This memory can be deallocated once the computation \eqref{eq:lp1:explicit} is completed.
We remark that the memory consumption is considerably lower than other alternative QR-factorization methods, as for example the Givens rotations method \cite[Alg. 5.2.2]{golu:vanl:96}, which requires a memory consumption $O(M^2)$ to store the matrix $Q$. This reduced memory consumption is instrumental in the GPU approach, as we explain in forthcoming \S \ref{section:num:lp1}.

\subsection{Pseudo-algorithms for GPU}
\label{sec:GPU pseudo}
Algorithm \ref{alg:srmdp} will be implemented on an NVIDIA GPU device.
The device architecture is built around a scalable array of multithreaded Streaming Multiprocessors (SMs);
each multiprocessor is designed to execute hundreds of threads concurrently.
To manage such a large amount of threads, it employs a unique architecture called SIMT (Single-Instruction, Multiple-Thread).
The code execution unit is called a kernel and is executed simultaneously on all SMs by independent blocks of threads.
Each thread is assigned to a single processor and executes within its own execution environment.
Thus, all threads run the same instruction at a time, although over different data.
In this section we briefly describe pseudo-codes for the Algorithm \ref{alg:srmdp}.

The algorithm has been programmed using the Compute Unified Device Architecture (CUDA) toolkit, specially designed for NVIDIA GPUs, see \cite{CUDA:programmingGuide}.
The code was built from an optimized \textrm{C} code.
The below pseudo-algorithms reflect this programming feature.
For the generation of the random numbers in parallel we took advantage of the NVIDIA CURAND library, see \cite{CUDA:curandLibrary}.

%The NVIDIA GPU architecture is built around a scalable array of multithreaded Streaming Multiprocessors (SMs). A multiprocessor is designed to execute hundreds of threads concurrently. To manage such a large amount of threads, it employs a unique architecture called SIMT (Single-Instruction, Multiple-Thread). The code execution unit is called a kernel and is executed simultaneously on all SMs by independent blocks of threads; each thread is assigned to a single processor and executes within its own execution environment, they all run the same instruction at a time, although over different data.

%In this section we briefly describe pseudo-codes for the Algorithm \ref{alg:srmdp}.
The time loop corresponding to the backward iteration of  Algorithm \ref{alg:srmdp} is shown in  Listing \ref{list:timeLoop};
the kernel corresponds to the use of either the \LPz \ or the \LPo \ basis.
In Listing \ref{list:kernelLP0}, a sketch for the  \LPz \ kernel is given. Notice that we are paralellizing the loop for any stratum index $k \in \{1,\ldots,K\}$ in the Algorithm \ref{alg:srmdp};
the terms $\ObsM{Y,i}{\vecx{x}_i}$ and $\ObsM{Z,i}{w,\vecx{x}_i}$ are computed in the \texttt{compute\_responses\_i} function, and the coefficients for  $\psim_{Y,i,k}(\cdot)$ and $\psim_{Z,i,k}(\cdot)$ are computed in \texttt{compute\_psi\_Y} and \texttt{compute\_psi\_Z}, respectively, according to \eqref{eq:lp0:explicit}.
%Therefore, each thread simulates the forward process, computes the responses $\ObsM{Y,i}{\vecx{x}_i}$ and $\ObsM{Z,i}{w,\vecx{x}_i}$, and the regression coefficients $\psim_{Y,i,k}(\cdot)$ and $\psim_{Z,i,k}(\cdot)$.
Having in view an optimal performance, matrices storing the simulations, responses and regression coefficients are fully interleaved, thus allowing coalesced memory accesses, see \cite{CUDA:programmingGuide}. Note that all device memory accesses are coalesced except those accesses to the regression coefficients in the resimulation stage during the computation of the responses, because one is not certain in which hypercube up each path will land.
In Listing \ref{list:kernelLP1}, the sketch for the  \LPo \ kernel   is given. Additionally to the tasks of the kernel in Listing \ref{list:kernelLP0}, each thread builds the matrix $A$ and applies a QR factorization, as detailed in \S \ref{sec:explicit coefs}. Note that in addition to the matrices just explained in \LPz, the matrix $A$ is fully interleaved thus allowing fully coalesced accesses. The global memory is allocated at the beginning of the program and is freed at the end, thus allowing kernels to reuse already-allocated memory wherever possible. In addition to global memory, kernels are also using local memory, for example for storing the resimulated forward paths used for computing the responses.
The coefficients for $\psim_{Y,i,k}(\cdot)$ and $\psim_{Z,i,k}(\cdot)$ are computed according to \eqref{eq:lp1:explicit}.
% Afterwards, using the Givens method \cite[Alg. 5.2.2]{golu:vanl:96} a batch QR-factorization of $A$ is performed, so that each thread computes its local QR factorization. Furthermore, each thread computes the regression coefficients following the formulas detailed in the Section \ref{sec:explicit coefs}.

\begin{lstlisting}[style=C,frame=single,label=list:timeLoop,caption={Backward iteration for $i=N-1$ to $i=0$}.]
int i
curandState *devStates
Initialize devStates
Initialize n_blocks, n_threads_per_block
for(i=N-1; i>=0; i--)
	kernel_bsde<<<n_blocks,n_threads_per_block>>>(i, devStates, ...)
\end{lstlisting}

\begin{lstlisting}[style=C,frame=single,label=list:kernelLP0,caption={Kernel for the approximation with \LPz}.]
__global__ void kernel_bsde_LP0(int i, curandState* devStates, ...) {
	const unsigned int global_tid = blockDim.x * blockIdx.x + threadIdx.x
	curandState localState = devStates[global_tid]

	unsigned long long int bin
	for(bin=global_tid; bin<K; bin+=n_blocks*n_threads_per_block) {
		simulates_x(&localState, global_tid, bin, ...)

		compute_responses_i(&localState, global_tid, i, ...)

		compute_psi_Z(global_tid, bin, i, ...)

		compute_psi_Y(global_tid, bin, i, ...)
	}
	devStates[global_tid] = localState
}
\end{lstlisting}

\begin{lstlisting}[style=C,frame=single,label=list:kernelLP1,caption={Kernel for the approximation with \LPo}.]
__global__ void kernel_bsde_LP1(int i, curandState *devStates, ...) {
	const unsigned int global_tid = blockDim.x * blockIdx.x + threadIdx.x
	curandState localState = devStates[global_tid]

	unsigned long long int bin
	for(bin=global_tid;bin<K;bin+=n_blocks*n_threads_per_block) {
		simulates_x(&localState, global_tid, bin, ...)

		compute_responses_i(&localState, global_tid, i, ...)

		build_d_A(global_tid, d_A, ...)

		qr(global_tid, d_A, ...)

		compute_psi_Z(global_tid, bin, i, d_A, ...)

		compute_psi_Y(global_tid, bin, i, d_A, ...)
	}
	devStates[global_tid] = localState
}
\end{lstlisting}

\subsection{Theoretical complexity analysis}
\label{subsection:theoreticalComplexityAnalysis}
In this section, we assume that the functions $y_i(\cdot)$ and $z_i(\cdot)$ are smooth, namely globally Lipschitz (resp. $C^{1}$ and the first derivatives are globally Lipschitz) in the \LPz\ (resp. \LPo) case.
The strata will be composed of uniform hypercubes of side length $\delta>0$ in the domain $[-L,L]^d$, where $L=\log(N)/\mu $ and $\mu$ is the parameter of the logistic distribution.  This choice ensures $\nu\left(\R^d\backslash [-L,L]^d\right)\leq 2d \exp(-  \mu  L)=O(N^{-1})$.
Our aim is to calibrate the numerical parameters (number of simulations and number of strata) so that the error given in Theorem \ref{thm:mc:err decomp} is $O(N^{-1})$, where $N$ is the number of time-steps. This tolerance error is the one we usually obtain after time discretization with $N$ time points \cite{zhan:04,gobe:makh:10,turk:13b}. In the following, we focus on polynomial dependency w.r.t. $N$, keeping only  the highest degree, ignoring constants and $\log(N)$ terms.
%The important numerical parameters we wish to observe are the dependence on the dimension, the number of time steps, the number of simulations and the number of strata used in the simulation.

\textit{Squared bias errors $T^{Y}_{1,i}$ and $T^{Z}_{1,i}$ in \eqref{eq:TY:1i}-\eqref{eq:TZ:1i}.} 
First, we remark that the approximation error of the numerical scheme, namely the error due to basis selection, depends principally on the size $\delta$ of strata.
In the case of \LPz, the squared bias error is proportional to the squared hypercube diameter plus the tail contribution, i.e. $O\left(\delta^2+\nu\left(\R^d\backslash[-L,L]^d\right)\right)$;
to calibrate this bias to $O(N^{-1})$, we require $\delta = O(N^{-1/2})$.
In contrast, the squared bias in $[-L,L]^d$ using \LPo \ is proportional to the fourth power  of the hypercube diameter, whence $\delta = O(N^{-1/4})$.
As a result, ignoring the $\log$ terms the number of required hypercubes is
\begin{align}
&\LPz: \quad K =  O(N^{d/2}), &
& \LPo: \quad K = O(N^{d/4}),
\end{align}
in both cases.

\paragraph{Statistical and interdependence errors}
These error terms depend on the number of local polynomials, as well as on the number of simulations.
Indeed, denoting $K'=\DimLinSpace{Y\text{ or }Z,.}$ the number of local polynomials and $M$ the number of simulations in the hypercube, then both errors are dominated by $O\left( N K' \log(M) / M\right)$, where the factor  $N$ comes from the $Z$ part of the solution (see $\cE(i)$ in Theorem \ref{thm:mc:err decomp}).
For \LPz\ (resp. \LPo), $K'=1$ or $q$ (resp. $K'=d+1$ or $q(d+1)$).
This implies to select
\begin{align}
&\LPz: \quad M = O( N^2) , &
&\LPo: \quad M = O(N^2) ,
\end{align}
again omitting the $\log$ terms.

\paragraph{Computational cost}\newcommand\Cost{{\cal C}_{{\rm cost}}}
The computational cost (in flops) of the simulations per hypercube is equal to $O(M \times N)$, because we simulate $M$ paths (of length $N$) of the process $X$.
The cost of the regression per hypercube is  $O\left( M\times N \right)$, see \S \ref{sec:explicit coefs}, and thus equivalent to the simulation cost.
Putting in the values of $M$ from the last paragraph,  the overall computational cost $\Cost$ (summed over all hypercubes and time steps) is
\begin{align}
&\LPz:  \quad \Cost^{\tt SEQ} = O(N^{ 4+d/2} ) , &
&\LPo: \quad \Cost^{\tt SEQ} = O( N^{ 4+d/4}) .
\end{align}
This quantity is related to the computational time for a sequential system (SEQ implementation) where there is no parallel computing.
For the GPU implementation, described in \S \ref{sec:GPU pseudo}, there is an additional computational time improvement since the computations on the hypercubes will be threaded across the cores of the card. Thus, the computational cost on GPU is
\begin{align}
&\LPz:  \quad \Cost^{\tt GPU} = O(N^{ 4+d/2} ) /\LoadF, &
&\LPo: \quad \Cost^{\tt GPU} = O( N^{ 4 +d/4})/\LoadF .
\end{align}
where the load factor $\LoadF$ is ideally the number of threads on the device.

Finally, we quantify the improvement in memory consumption offered by the SRMDP algorithm compared to the LSMDP algorithm of \cite{gobe:turk:14}.
This is a very important improvement, because, as explained in the introduction, the memory is the key constraint in solving problems in high dimension.
We only compare sequential versions of the algorithms, meaning that the computational costs will be the same.
The main difference between the two schemes is then in the number of simulations that must be stored in the machine at any given time.
We summarize this in Table \ref{table:cost} below.
\begin{table}[htb]
\begin{center}
\footnotesize
\begin{tabular}{|c|c|c|c|c|c|}
\hline
Algorithm &\multicolumn{2}{c|}{Number of}& \multicolumn{2}{c|}{Computational }\\
 & \multicolumn{2}{c|}{simulations}& \multicolumn{2}{c|}{ cost }\\
 & \LPz & \LPo&\LPz & \LPo\\
\hline SRMDP &$N^{2}$&$N^{2}$&$N^{ 4+d/2}$&$N^{ 4+d/4}$\\
LSMDP &$N^{2+d/2}$&$N^{2+d/4}$&$N^{4+d/2}$&$N^{4+d/4}$\\
\hline
\end{tabular}
\end{center}
\caption{\label{table:cost}Comparison of numerical parameters with or without stratified sampling, as a function of $N$.}
\end{table}

In SRMDP, the memory consumption is mainly related to storing  coefficients representing the solutions on hypercubes, that is $O( N\times \DimLinSpace{Y\text{ or }Z,.}\times K)$;
if one is using the  \LPo \ basis, one must also take into account the memory consumption per strata $M\times(d+1) = O(N^2)$ for the QR factorization, explained in \S \ref{sec:explicit coefs}.
In contrast, the memory consumption for LSMDP is mainly $O(K \times N^2)$, which represents the number of simulated paths of the Markov chains that must be stored in the machine at any given time.
We summarize the memory consumption of the two algorithms in Table \ref{table:memory}.

\begin{table}[htb]
\begin{center}
\footnotesize{
\begin{tabular}{|c|c|c|}
\hline
Algorithm & \LPz & \LPo\\
\hline
SRMDP &$N^{ 1+d/2}$ &$ N^{ 1+d/4}\vee N^2$\\
LSMDP & $N^{ 2+d/2}$ &$N^{ 2+d/4}$\\
\hline
\end{tabular}
}
\end{center}
\caption{\label{table:memory}Comparison of memory requirement as a function of $N$.}
\end{table}
Observe that SRMDP requires $N$ times less memory than LSMDP with the \LPz \ basis.
This implicitly implies a gain of 2 on the dimension $d$ that can be handled.
On the other hand, if the \LPo \ basis is used, the SRMDP requires $O(N^{d/4})$ less memory for $d\le 4$ than LSMDP, and $N$ times less memory for $d\ge 4$.
Therefore, there is an implicit gain of $4$ in the dimension that can be handled by the algorithm.

\section{Numerical experiments}
\label{section:numerics}

\subsection{Model, stratification, and performance benchmark}
\label{num:model}
We use the Brownian motion model $X=W$ ($d=q$). Moreover, 
 the numerical experiments will consider the performance according to the dimension $d$.
 We introduce the function $\omega(t,x) = \exp(t + \sum_{ k=1}^q x_{ k})$.
We perform numerical experiments on the BSDE with data $g(x) =\omega(T,x)(1+\omega(T,x))^{-1}$ and
\[f(t,x,y,z) = \left( \sum_{k=1}^q z_k \right) \left(y- {2+q \over 2q} \right),\]
where $z= (z_1,\ldots,z_q)$.
The BSDE has explicit solutions in this framework, given by
\[
y_i(x) =\omega(t_i,x)(1+\omega(t_i,x))^{-1},\qquad z_{k,i}(x)  = \omega(t_i,x)(1+\omega(t_i,x))^{-2},
\]
where $z_{ k,i}(x)$ is the $ k$-th component of the $q$-dimensional cylindrical function $z_{i}(x) \in \R^q$.

The logistic distribution for Algorithm \ref{alg:srmdp} is parameterized by $\mu = 1$ and we consider $T=1$.
For the least-squares Monte Carlo, we stratify the domain $[-6.5,6.5]^q$ with uniform hypercubes.
To assess the performance of the algorithm, we compute the average mean squared error (MSE) over $10^3$ independent runs of the algorithm for three error  indicators:
\begin{equation}
\begin{split}
MSE_{Y,\text{max}}&:= \ln \left\{ 10^{-3} \max_{0\le i \le N-1} \sum_{m=1}^{10^3} | y_i(R_{i,m}) - \yM i{R_{i,m}}|^2 \right\} ,\\
MSE_{Y,\text{av}}&:= \ln \left\{ 10^{-3} N^{-1} \sum_{m=1}^{10^3} \sum_{i=0}^{N-1}  | y_i(R_{i,m}) - \yM i{R_{i,m}}|^2 \right\} ,\\
MSE_{Z,\text{av}}&:= \ln \left\{ 10^{-3}  N^{-1} \sum_{m=1}^{10^3} \sum_{i=0}^{N-1}  | z_i(R_{i,m}) - \zM i{R_{i,m}}|^2 \right\} ,
\end{split}
\label{mse}
\end{equation}
where the simulations $\{R_{i,m}; i = 0,\ldots,N-1, \ m=1,\ldots,10^3\}$ are independent and identically $\nu$-distributed,  and independently drawn from the simulations used for the LSMC scheme.
We parameterize the hypercubes according to the instructions given in the theoretical complexity analysis, see \S \ref{subsection:theoreticalComplexityAnalysis}.
In particular, we consider different values of $N$ and always set $K =O(N^{d/2})$ in \LPz\ (resp. $O(N^{d/4})$ in \LPo) and $M = O( N^2)$.
Note, however, that we do not specify the value of $\delta$, but rather the number of hypercubes per dimension $K^{1/q}$, which we denote \texttt{\#C} in what follows;
this being equivalent to setting $\delta$, but is more convenient to program.
As we shall illustrate, the error converges as predicted as $N$ increases, although the exact error values will depend on the constants that we choose in the parameterization of $K$ and $M$.

\subsection{CPU and GPU performance}
In this section, several experiments based on \S \ref{num:model} are presented to assess the  performance of CUDA implementation of Algorithm \ref{alg:srmdp};
the pseudo-algorithms are given in \S \ref{sec:GPU pseudo}.
We shall compare  its performance with a version of SRMDP implemented to run on multicore CPUs. For the design of this comparison we have followed some ideas in \cite{lee:kim:10}. Moreover, in order to test the theoretical results of \S \ref{subsection:theoreticalComplexityAnalysis},
we compare the performance of the two algorithms according to the choice of the basis functions, the impact of this choice on the convergence of the approximation of the BSDE, and the impact of this choice on the computational performance in terms of computational time and memory consumption.

There are two types of basis functions we investigate: \LPz \ in \S \ref{section:lp0}, and \LPo \ in \S \ref{section:num:lp1}.
As explained in \S \ref{subsection:theoreticalComplexityAnalysis}, the \LPz \ basis is highly suited to GPU implementation because it has a very low memory requirement per thread of computation.
On the other hand, it has a very high global memory requirement for storing coefficients.
This represents a problem in high dimensions because one needs many coefficients to obtain a good accuracy.
On the other hand, the \LPo \ basis involves a higher cost per thread, although requires a far lower global memory for storing coefficients; this implies that the impact of the GPU implementation is lower in moderate dimensional problems, but that one can solve problems in higher dimension.
Moreover, the full performance impact of the GPU implementation on the \LPo \ basis is in high dimension, where the number of strata is very high and therefore the GPU is better saturated with computations.
We illustrate numerically all of these effects in the following sections.

The numerical experiments have been performed with the following hardware and software configurations: a GPU GeForce GTX TITAN Black with $6$ GBytes of global memory (see \cite{CUDA:keplerArchitecture} for details in the architecture), two recent multicore Intel Xeon CPUs E5-2620 v2 clocked at $2.10$ GHz ($6$ cores per socket) with $62$ GBytes of RAM, CentOS Linux, NVIDIA CUDA SDK 7.5 and INTEL C compiler 15.0.6.
The CPU programs  were optimized and parallelized using OpenMP \cite{openmp}.
Since the CUDA code has been derived from an optimized C code, both codes perform the same algorithms,  and their performance can be fairly compared according to computational times;
the multicore CPUs time (CPU) and the GPU time (GPU) will all be measured in seconds in the forthcoming tables. CPU times correspond to executions using $24$ threads so as to take advantage of Intel Hyperthreading. The speedups of the CPU parallel version with respect to pure sequential CPU code are around $16$. The results are obtained in single precision, both in CPU and GPU.

\subsubsection{Examples with the approximation with \LPz \ local polynomials}
\label{section:lp0}

%We present a list of tables with $d=4$, $6$ and $11$.
All examples  will be run using $64$ thread blocks each with $256$ threads.
In Table \ref{table:LP0d4} we show  results for $d=4$, with \texttt{\#C}=$\left\lfloor{4\sqrt{N}}\right\rfloor$ and $M = N^2$.
%As in all forthcoming tables, note that convergence properties are illustrated by the fact that when decreasing $\Delta_t$, increasing \texttt{\#C} and $M$ the values of the error indicators become more negative. Moreover,
%an increasing computational cost to improve convergence is observed and
Except for the cases $\Dt = 0.2$ and $\Dt = 0.1$ where there are not enough strata to fully take advantage of the GPU, the GPU implementation provides a significant reduction in the computational time:
the GPU speed-up reaches the value $18.90$.
Moreover, the speed-up improves as we increase the \texttt{\#C}.

\begin{table}[h]
\begin{center}
{\footnotesize
\begin{tabular}{|r|r|r|r||r|r|r|r|r|r|r|}
\hline
$\Delta_t$ & \texttt{\#C} & $K$ & $M$ & $MSE_{Y,\text{max}}$ & $MSE_{Y,\text{av}}$ & $MSE_{Z,\text{av}}$ & CPU & GPU \\
\hline
\hline
 $0.2$ & $8$ & $4096$ & $25$ & $-3.712973$ & $-3.774071$ & $-0.964842$ & $0.23$ & $2.00$ \\ %
%-3.712973 -3.774071 -0.964842 0.694381 1.609438
\hline
 $0.1$ & $12$ & $20736$ & $100$ & $-4.066741$ & $-4.303750$ & $-1.607104$ & $5.23$ & $2.20$ \\ % 2.52
%-4.066741 -4.303750 -1.607104 0.788880 2.302585
\hline
 $0.05$ & $17$ & $83521$ & $400$ & $-4.337988$ & $-4.698645$ & $-2.302092$ & $171.92$ & $12.39$ \\ %
%-4.337988 -4.698645 -2.302092 2.516648 2.995732
\hline
 $0.02$ & $28$ & $614656$ & $2500$ & $-4.472564$ & $-4.988069$ & $-3.225411$ & $58066.33$ & $3070.92$ \\ %
%-4.472564 -4.988069 -3.225411 8.029731 3.912023
\hline
\end{tabular}
}

\end{center}
\caption{\LPz \ local polynomials, $d=4$, \texttt{\#C}=$\left\lfloor{4\sqrt{N}}\right\rfloor$, $M = N^2$.}
\label{table:LP0d4}
\end{table}

Tables \ref{table:LP0d6_0} and \ref{table:LP0d6_1} show results for $d=6$ with \texttt{\#C}=$\left\lfloor{\sqrt{N}}\right\rfloor$ and \texttt{\#C}=$\left\lfloor{2\sqrt{N}}\right\rfloor$, respectively.
Convergence is clearly improved by doubling \texttt{\#C}.
In Table \ref{table:LP0d6_1} the case of $\Delta_t=0.02$ is not shown due to insufficient GPU global memory \footnote{For $\Delta_t=0.02$, the array for storing the regression coefficients will be of size $N \times K \times (D+1)$, i.e. $50 \times 14^6 \times 7 \times 4$ bytes using single precision, which equals $9.81$ GBytes, much greater than the available $6$ GBytes of device memory.}. 
In Table \ref{table:LP0d6_0}, the GPU speed-up reaches $15.93$, whereas in Table \ref{table:LP0d6_1} it reaches $14.85$. As in Table \ref{table:LP0d4}, the increase  in the speed-up is explained due to the increased number of hypercubes, thus demonstrating how important it is to have many hypercubes in the GPU implementation.
However, the finer basis requires $2^6$ times as much memory for storing coefficients.

\begin{table}[h]
\begin{center}
{\footnotesize
\begin{tabular}{|r|r|r|r||r|r|r|r|r|r|r|}
\hline
$\Delta_t$ & \texttt{\#C} & $K$ & $M$ & $MSE_{Y,\text{max}}$ & $MSE_{Y,\text{av}}$ & $MSE_{Z,\text{av}}$ & CPU & GPU\\
\hline
\hline
 $0.2$ & $2$ & $64$ & $25$ & $-2.392320$ & $-2.451332 $ & $-0.431059$ & $0.21$ & $1.99$ \\
% -2.392320 -2.451332 -0.431059 0.687446 1.609438
\hline
 $0.1$ & $3$ & $729$ & $100$ & $-2.440274$ & $-2.500775$ & $-1.096603$ & $0.47$ & $2.05$ \\
% -2.440274 -2.500775 -1.096603 0.718570 2.302585
\hline
 $0.05$ & $4$ & $4096$ & $400$ & $-2.829757$ & $-2.905192$ & $-1.687142$ & $17.21$ & $3.15$ \\
% -2.829757 -2.905192 -1.687142 1.146564 2.995732
\hline
 $0.02$ & $7$ & $117649$ & $2500$ & $-3.235130$ & $-3.539011$ & $-2.557686$ & $13930.70$ & $874.25$ \\
% -3.235130 -3.539011 -2.557686 6.773372 3.912023
\hline
\end{tabular}
}
\end{center}
\caption{\LPz \ local polynomials, $d=6$, \texttt{\#C}=$\left\lfloor{\sqrt{N}}\right\rfloor$, $M = N^2$.}
\label{table:LP0d6_0}
\end{table}

\begin{table}[h]
\begin{center}
{\footnotesize
\begin{tabular}{|r|r|r|r||r|r|r|r|r|r|}
\hline
$\Delta_t$ & \texttt{\#C} & $K$ & $M$ & $MSE_{Y,\text{max}}$ & $MSE_{Y,\text{av}}$ & $MSE_{Z,\text{av}}$ & CPU & GPU\\
\hline
\hline
 $0.2$ & $4$ & $4096$ & $25$ & $-2.707882$ & $-2.784022$ & $-0.477751$ & $0.29$ & $1.94$ \\
% -2.707882 -2.784022 -0.477751 0.664581 1.609438
\hline
 $0.1$ & $6$ & $46656$ & $100$ & $-3.195937$ & $-3.294488$ & $-1.133834$ & $13.72$ & $2.44$ \\
% -3.195937 -3.294488 -1.133834 0.892564 2.302585
\hline
 $0.05$ & $8$ & $262144$ & $400$ & $-3.505867$ & $-3.664396$ & $-1.795697$ & $775.33$ & $52.20$ \\
% -3.505867 -3.664396 -1.795697 3.955182 2.995732
\hline
\end{tabular}
}

\end{center}
\caption{\LPz \ local polynomials, $d=6$, \texttt{\#C}=$\left\lfloor{2\sqrt{N}}\right\rfloor$, $M = N^2$.}
\label{table:LP0d6_1}
\end{table}

Table \ref{table:LP0d11} shows that the algorithm can work for $d=11$ in several seconds with a reasonable accuracy in a GPU.
The corresponding speed-up with respect to CPU version is around $13.35$. For the execution with $\Delta_t = 0.1$ we are going to report the GFlop rate, and also the memory transfer to/from the global memory. Inside the kernel, the functions computing the regression coefficients (denoted by \texttt{compute\_psi\_Z} and \texttt{compute\_psi\_Y} in the Listing \ref{list:kernelLP0}) are memory bounded, reaching  236.795 GBytes/s when reading/writing from/to global memory. The rest of the kernel is more computationally limited. In the overall kernel, the memory transfer from/to global memory is around 160 GBytes/s and the Gflop rate is around 238 GFlop/s, although around the 30\% of the instructions executed by the kernel are integer instructions to assign the simulations to the strata in the resimulation stage during the computation of the responses.

\begin{table}[h]
\begin{center}
{\footnotesize
\begin{tabular}{|r|r|r|r||r|r|r|r|r|r|r|}
\hline
$\Delta_t$ & \texttt{\#C} & $K$ & $M$ & $MSE_{Y,\text{max}}$ & $MSE_{Y,\text{av}}$ & $MSE_{Z,\text{av}}$ & CPU & GPU \\
\hline
\hline
 $0.2$ & $2$ & $2048$ & $25$ & $-2.152253$ & $-2.202357$ & $0.211590$ & $0.27$ & $1.99$ \\
% -2.152253 -2.202357 0.211590 0.689324 1.609438
\hline
 $0.1$ & $3$ & $177147$ & $100$ & $-2.144843$ & $-2.267742$ & $-0.469759$ & $67.96$ & $6.29$ \\
% -2.144843 -2.267742 -0.469759 1.840321 2.302585
\hline
 $0.05$ & $4$ & $4194304$ & $400$ & $-2.484169$ & $-2.633602$ & $-1.070096$ & $28154.07$ & $2108.64$ \\
% -2.484169 -2.633602 -1.070096 7.653799 2.995732
\hline
\end{tabular}
}

\end{center}
\caption{\LPz \ local polynomials, $d=11$, \texttt{\#C}=$\left\lfloor{\sqrt{N}}\right\rfloor$, $M = N^2$.}
\label{table:LP0d11}
\end{table}

\subsubsection{Examples with the approximation with \LPo \ local polynomials}
\label{section:num:lp1}

In this section we show the results corresponding to the approximation with the \LPo \ basis.
Compared to \LPz, this basis consumes much less global memory to store  coefficients, because it requires far fewer hypercubes, see \S \ref{subsection:theoreticalComplexityAnalysis}.
On the other hand, the approximation with \LPo \ basis demands higher thread memory due to the storage of a large matrix for each hypercube, as explained in \S \ref{subsection:theoreticalComplexityAnalysis}.
This may have an impact on the computational time on the GPU:
recalling from \S \ref{sec:GPU pseudo} the GPU handles multiple hypercubes at any given time, each one requiring the storage of a matrix $A$, the global memory capacity of the GPU device restricts the number of threads we can handle at any given time.
This issue is much less significant with the \LPz \ basis.
In order to optimize the performance of the \LPo \ basis, we must minimize the thread memory storage.
 We implement the Householder reflection method for QR-factorization, \cite[Alg. 5.3.2]{golu:vanl:96}.
For this, we must store a matrix containing $M\times(d+1) = O( N^2)$ floating point values per thread on the GPU memory.
Thanks to the reduced global memory storage for coefficients, we are able to work in a rather high dimension $d=19$.
\begin{remark}
There are many methods to implement QR-factorization.
However, the choice of method has a substantial impact on the performance of the GPU implementation.
For example, the Givens rotation method \cite[Alg. 5.2.2]{golu:vanl:96} requires the storage of an $M\times M$ matrix, which corresponds to $O(N^4)$ floating points.
This is rather more than the required $O(N^2)$ for the Householder reflection method given in Section \ref{sec:explicit coefs}.
Therefore, the Givens rotation method would be far slower when implemented on a GPU than the Householder reflection method, because it may not be possible to use an optimal thread configuration.
\end{remark}

\begin{remark}
In the forthcoming examples, we use more simulations per stratum for the \LPo \ basis compared to the equivalent results for \LPz.
This is to account for the additional statistical and interdependence errors, 
as explained in \S \ref{subsection:theoreticalComplexityAnalysis}.
\end{remark}

 In Table \ref{table:LP1d4}, we present results for $d=4$. These results are to be compared with Table \ref{table:LP0d4}, where in particular the $MSE_{Z,\text{av}}$ results are closer line to line.
 The computational time is substantially improved for the CPU and GPU calculations.
 Also note that, unlike  for the $Z$ component, the accuracy for the $Y$ component is substantially better for the \LPo \ basis than for the \LPz \ one.
 The difference in the accuracy results between the $Y$ and $Z$ components is likely explained by the fact that the function $x\mapsto z_i(x)$ is rather flat, so it is much better approximated by \LPz \ basis functions than $x\mapsto y_i(x)$. The GPU speed-up reaches $8.05$.

 \begin{table}[h]
\begin{center}
{\footnotesize
\begin{tabular}{|r|r|r|r||r|r|r|r|r|r|r|r|r|}
\hline
$\Delta_t$ & \texttt{\#C} & $K$ & $M$ & $MSE_{Y,\text{max}}$ & $MSE_{Y,\text{av}}$ & $MSE_{Z,\text{av}}$ & CPU & GPU \\
\hline
\hline
 $0.2$ & $3$ & $81$ & $125$ & $-4.021483$ & $ -4.131725$ & $ -0.900286$ & $0.11$ & $0.23$ \\
% -4.021483 -4.131725 -0.900286 0.824924 1.609438
\hline
 $0.1$ & $5$ & $625$ & $500$ & $-4.290881$ & $-4.695769$ & $-1.551480$ & $1.26$ & $0.79$ \\
% -4.290881 -4.695769 -1.551480 1.242652 2.302585
\hline
 $0.05$ & $7$ & $2401$ & $2000$ & $-4.541253$ & $-5.022405$ & $-2.281332$ & $43.56$ & $7.83$ \\
% -4.541253 -5.022405 -2.281332 3.145233 2.995732
\hline
 $0.02$ & $10$ & $10000$ & $12500$ & $-4.574551$ & $-5.143310$ & $-3.228237$ & $6827.98$ & $847.83$ \\
% -4.574551 -5.143310 -3.228237 7.868967 3.912023
\hline
\end{tabular}}
\end{center}
\caption{\LPo \ local polynomials, $d=4$, \texttt{\#C}=$\left\lfloor{3\sqrt{d\sqrt{N}}}-5\right\rfloor$, $M = (d+1)N^2$.}
\label{table:LP1d4}
\end{table}
 
Next, results for $d=6$ are shown. Thus, we compare Table \ref{table:LP1d6_1} below with Table \ref{table:LP0d6_1}.
For a given precision on the $Z$ component of the solution, we observe substantial improvements in the CPU codes, but no such gains on the GPU version.
In contrast, the accuracy of the $Y$ approximation is, as in the $d=4$ case, substantially better.
Moreover, whereas we were not able to do computations for $\Dt = 0.02$ with the \LPz \ basis due to insufficient GPU memory, we are now able to make these calculations with the  \LPo \ basis.
The GPU speed-up reaches $6.13$, which is lower than the \LPz \ basis speed-up factor, as expected.

 \begin{table}[h]
\begin{center}
{\footnotesize
\begin{tabular}{|r|r|r|r||r|r|r|r|r|r|r|r|r|}
\hline
$\Delta_t$ & \texttt{\#C} & $K$ & $M$ & $MSE_{Y,\text{max}}$ & $MSE_{Y,\text{av}}$ & $MSE_{Z,\text{av}}$ & CPU & GPU \\
\hline
\hline
 $0.2$ & $2$ & $64$ & $175$ & $-3.504153$ & $-3.668801$ & $-0.461077$ & $0.20$ & $0.32$ \\
% -3.504153 -3.668801 -0.461077 0.897835 1.609438
\hline
 $0.1$ & $3$ & $729$ & $700$ & $-3.804091$ & $-3.911488 $ & $-1.133263$ & $1.84$ & $1.66$ \\
% -3.804091 -3.911488 -1.133263 1.631043 2.302585
\hline
 $0.05$ & $4$ & $4096$ & $2800$ & $-4.075928$ & $-4.231639$ & $-1.791519$ & $125.81$ & $20.50$ \\
% -4.075928 -4.231639 -1.791519 4.070794 2.995732
\hline
 $0.02$ & $6$ & $46656$ & $17500$ & $-3.809734$ & $-4.529827$ & $-2.689432$ & $82529.21$ & $15283.18$ \\
% -3.809734 -4.529827 -2.689432 10.686154 3.912023 
\hline
\end{tabular}}
\end{center}
\caption{\LPo \ local polynomials, $d=6$, \texttt{\#C}=$\left\lfloor{1.5\sqrt{d\sqrt{N}}}-3\right\rfloor$, $M = (d+1)N^2$.}
\label{table:LP1d6_1}
\end{table}

In the high dimensional $d=11$ setting shown in Table \ref{table:LP1d11}, we compare with Table \ref{table:LP0d11}. We observe a speed-up of order $5.63$ compared to the CPU implementation.

 \begin{table}[h]
\begin{center}
{\footnotesize
\begin{tabular}{|r|r|r|r||r|r|r|r|r|r|r|r|r|}
\hline
$\Delta_t$ & \texttt{\#C} & $K$& $M$ & $MSE_{Y,\text{max}}$ & $MSE_{Y,\text{av}}$ & $MSE_{Z,\text{av}}$ & CPU & GPU \\
\hline
\hline
 $0.2$ & $2$ & $2048$ &  $2000$ & $-3.271648$ & $-3.368051 $ & $-1.455388$ & $10.33$ & $3.41$ \\
% -3.271648 -3.368051 -1.455388 2.086342 1.609438
\hline
 $0.2$ & $3$ & $177147$ & $4000$ & $-3.269004$ & $-3.403994 $ & $-1.975300$ & $1635.95$ & $290.56$ \\
% -3.269004 -3.403994 -1.975300 6.011379 1.609438
\hline
\end{tabular}}
\end{center}
\caption{\LPo \ local polynomials, $d=11$.}
\label{table:LP1d11}
\end{table}

In the remainder of this section, we present results in  dimension $d=12 $ to $d = 19$ (in Tables \ref{table:LP1d12}, \ref{table:LP1d13}, \ref{table:LP1d14} and \ref{table:LP1d15_19}, respectively) for which the capacity of the GPU is maximally used to provide the highest possible accuracy. The GPU speed-up reaches up to $5.67$ compared to the CPU implementation. Note that for the example with $d=19$ in Table \ref{table:LP1d15_19} the LSMDP algorithm would require $118$ GBytes of memory to store all the simulations at a given time, whereas the here proposed SRMDP algorithm can be executed with less than $6$ GBytes and with much less computational time owing to it does not need to associate the simulations to hypercubes. Finally, for the example with $\Delta_t = 0.2$, \texttt{\#C}$=2$ and $M=4000$ of Table \ref{table:LP1d13} we next report the GFlop rate and the memory transfer from/to global memory. In the overall kernel, the memory transfer from/to global memory is around 132 GBytes/s and the GFlop rate is around 136 GFlop/s. In order to understand why the device memory bandwidth used by LP1 kernel is lower than the one used by LP0 kernel, observe that at any given time, for each strata we are accessing $(d+1)$ times more elements in the LP1 framework in the re-simulation stage of the responses computation. 
Moreover, these accesses  are potentially non-coalesced, because the forward process is randomly re-simulated and we do not know a priori in which strata is going to fall.

\begin{table}[h]
\begin{center}
{\footnotesize
\begin{tabular}{|r|r|r|r||r|r|r|r|r|r|r|r|r|}
\hline
$\Delta_t$ & \texttt{\#C} & $K$ & $M$ & $MSE_{Y,\text{max}}$ & $MSE_{Y,\text{av}}$ & $MSE_{Z,\text{av}}$ & CPU & GPU \\
\hline
\hline
 $0.2$ & $2$ & $4096$ & $2000$ & $-3.111153$ & $-3.232051$ & $-1.297737$ & $22.29$ & $4.95$ \\
% -3.111153 -3.232051 -1.297737 2.306536 1.609438
\hline
 $0.2$ & $3$ & $531441$ & $4000$ & $-3.214096$ & $-3.272644$ & $-1.821935$ & $5554.49$ & $1196.28$ \\
% -3.214096 -3.272644 -1.821935 7.643454 1.609438
\hline
\end{tabular}}
\end{center}
\caption{\LPo \ local polynomials, $d=12$.}
\label{table:LP1d12}
\end{table}

\begin{table}[h]
\begin{center}
{\footnotesize
\begin{tabular}{|r|r|r|r||r|r|r|r|r|r|r|r|r|}
\hline
$\Delta_t$ & \texttt{\#C} & $K$ & $M$ & $MSE_{Y,\text{max}}$ & $MSE_{Y,\text{av}}$ & $MSE_{Z,\text{av}}$ & CPU & GPU \\
\hline
\hline
 $0.2$ & $2$ & $8192$ & $3000$ & $-2.995413$ & $-3.153302 $ & $-1.460911$ & $69.45$ & $12.46$ \\
%  -2.995413 -3.153302 -1.460911 2.920640 1.609438
\hline
 $0.2$ & $2$ & $8192$ & $4000$ & $-3.022855$ & $-3.158471$ & $-1.649632$ & $94.07$ & $16.58$ \\
% -3.022855 -3.158471 -1.649632 3.188144 1.609438
\hline
\end{tabular}}
\end{center}
\caption{\LPo \ local polynomials, $d=13$.}
\label{table:LP1d13}
\end{table}

\begin{table}[h]
\begin{center}
{\footnotesize
\begin{tabular}{|r|r|r|r||r|r|r|r|r|r|r|r|r|}
\hline
$\Delta_t$ & \texttt{\#C} & $K$ & $M$ & $MSE_{Y,\text{max}}$ & $MSE_{Y,\text{av}}$ & $MSE_{Z,\text{av}}$ & CPU & GPU \\
\hline
\hline
 $0.2$ & $2$ & $16384$ & $2000$ & $-3.011673$ & $-3.092870$ & $-1.026128$ & $102.11$ & $19.55$ \\
% -3.011673 -3.092870 -1.026128 3.309102 1.609438
\hline
 $0.2$ & $2$ & $16384$ & $4000$ & $-3.029663$ & $-3.105833$ & $-1.558935$ & $205.82$ & $50.62$ \\
% -3.029663 -3.105833 -1.558935 4.385373 1.609438
\hline
\end{tabular}}
\end{center}
\caption{\LPo \ local polynomials, $d=14$.}
\label{table:LP1d14}
\end{table}

\begin{table}[h]
\begin{center}
{\footnotesize
\begin{tabular}{|r|r|r||r|r|r|r|r|r|r|r|}
\hline
$d$ & $K$ & $M$ & $MSE_{Y,\text{max}}$ & $MSE_{Y,\text{av}}$ & $MSE_{Z,\text{av}}$ & CPU & GPU \\
\hline
\hline
$15$ & $32768$ & $5000$ & $-2.981181$ & $-3.106590$ & $-1.574532$ & $578.88$ & $139.60$ \\
\hline
$16$ & $65536$ & $6000$ & $-2.795353$ & $-2.959375$ & $-1.588716$ & $1411.75$ & $429.53$ \\
\hline
$17$ & $131072$ & $5000$ & $-2.772595$ & $-2.936549$ & $-1.371146$ & $2580.06$ & $793.61$ \\
\hline
$18$ & $262144$ & $4000$ & $-2.845755$ & $-2.918057$ & $-1.114600$ & $4275.13$ & $1589.30$ \\
\hline
$19$ & $524288$ & $3200$ & $-2.726427$ & $-2.851617$ & $-0.839849$ & $7245.91$ & $4370.31$ \\
\hline
\end{tabular}}
\end{center}
\caption{\LPo \ local polynomials, $d=15, \ldots, 19$, $\Delta_t = 0.2$, \texttt{\#C} $=2$.}
\label{table:LP1d15_19}
\end{table}

\appendix
\section*{Appendix}
\subsection{Proof of Proposition \ref{prop:hnu}}
\label{proof:proposition:prop:hnu}

It is known from \cite[Proposition 3.1]{gobe:turk:15} that it is sufficient to show that there is a  continuous $C_\rho : \R \to [1,\infty)$ such that,
for all $\Lambda \ge 0$, $\lambda \in [0,\Lambda]$, and $y\in \R^d$,
\begin{equation}
\label{eq:USES}
{\plog(y)  \over C_\rho(\Lambda) } \le \int_{\R^d} \plog (y + z \sqrt \lambda) \exp(-\frac {|z|^2}2)  \dd z \le C_\rho(\Lambda) \plog (y).
\end{equation}

The proof is given for $d=1$; generalization to higher dimensions is obvious because the multidimensional density is  just the product of the one-dimensional densities over the components. Moreover, for simplicity the proof is given for $\mu =1$, as generality in this parameter does not change the proof. For simplicity, we will write $\plog(x) = p(x)$ in what follows.

In terms of the hyperbolic cosine function, the density can be expressed as
\[
p(x) = {\exp(-x) \over \big(1 + \exp(-x)\big)^2} = \left( \exp(\frac x2) + \exp(-\frac x2)  \right)^{-1} = \left(2 \cosh(\frac x2) \right)^{-1}.
\]
We define $I(y,\lambda) := 2 \int_\R p(y+z\sqrt \lambda) \exp(-\frac{z^2}2) \dd z $, so that from the relation $\cosh(x+y) = \cosh(x)\cosh(y) + \sinh(x)\sinh(y)$, we have that
\begin{align*}
I(y,\lambda) = \int_{\R} { \exp(-\frac{z^2}2) \over  \cosh(\frac y2)\cosh(\frac{z\sqrt\lambda}2) + \sinh(\frac y2)\sinh(\frac{z\sqrt\lambda}2)} \dd z:=I_{+}(y,\lambda) +I_{-}(y,\lambda)
\end{align*}
where $I_{+,-}$ denotes respectively the integral on $\R^{+}$ and $\R^{-}$.
\paragraph{Upper bound}
Suppose that $y\ge0$. Then, if $z\ge0$, it follows that $\sinh(y/2)\sinh(z\sqrt\lambda/2) \geq 0$, whence
\begin{align*}
I_{+}(y,\lambda) &\le \int_{\R_+}{ \exp(-\frac{z^2}2) \over  \cosh(\frac y2) \cosh(\frac{z\sqrt\lambda}2)} \dd z  = 2  \int_{\R_+} e^{-\frac{z^2}2} \dd z \times p(y).
\end{align*}
On the other hand, if $z\le 0$, then $\sinh(\frac y2)\sinh(\frac{z\sqrt\lambda}2)\geq \cosh(\frac y2)\sinh(\frac{z\sqrt\lambda}2)$, therefore
\begin{align*}
I_{-}(y,\lambda)  &\le \int_{\R_-}{ \exp(-\frac{z^2}2) \over  \cosh(\frac y2) \{ \cosh(\frac{z\sqrt\lambda}2) + \sinh(\frac{z\sqrt\lambda}2)\}} \dd z
 = 2  \int_{\R_-} \exp\left(-\frac{z^2}2 - \frac{z\sqrt\lambda}2 \right) \dd z \times p(y).
\end{align*}
Therefore, if $y\geq 0$ then $I(y,\lambda) \le 2 \int_\R \exp ( \frac{-z^2 + (z)_{-} \sqrt\Lambda }2) \dd z \times p(y)$.
Observing that $I(y,\lambda)$ is symmetric in $y$, thus the upper bound \eqref{eq:USES} is proved.

\paragraph{Lower bound}
Suppose that $y\ge0$.
For $z\leq 0$, observe that $\sinh(\frac y2)\sinh(\frac{z\sqrt\lambda}2) \leq 0$, whence
\begin{align*}
I_{-}(y,\lambda)  &\ge \int_{\R_-}{ \exp(-\frac{z^2}2) \over  \cosh(\frac y2) \cosh(\frac{z\sqrt\lambda}2)} \dd z
 \geq  2  \int_{\R_-} \frac{\exp(-\frac{z^2}2)  }{ \cosh(\frac{z\sqrt\Lambda}2)} \dd z \times p(y).
\end{align*}
For $z\ge0$, we use that $\sinh(\frac y2)\sinh(\frac{z\sqrt\lambda}2) \le
\cosh(\frac y2)\sinh(\frac{z\sqrt\lambda}2)$ to obtain
\begin{align*}
I_{+}(y,\lambda) &\ge  \int_{\R_+}{ \exp(-\frac{z^2}2) \over  \cosh(\frac y2) \{ \cosh(\frac{z\sqrt\lambda}2) + \sinh(\frac{z\sqrt\lambda}2)\}} \dd z
  \\ 
&\ge 2  \int_{\R_+} \exp\left(-\frac{z^2}2 - \frac{z\sqrt\Lambda}2 \right) \dd z \times p(y).
\end{align*}
The result on $y<0$ follows again from the symmetry of $I(y,\lambda)$.
 \qed
 
\subsection{Stability results for discrete BSDE}
\label{section:apriori}
We recall standard results borrowed to \cite{gobe:turk:14} and adapted to our setting, they  are aimed at comparing two solutions of discrete BSDEs of the form \eqref{eq:MDP:intro} with different data. Namely, consider two discrete BSDEs,
$(Y_{1,i},Z_{1,i})_{0 \le i < N}$
and $(Y_{2,i}, Z_{2,i})_{0\le i< N}$, given by
\begin{align*}
 Y_{l,i} &= \Espi{ g(X_N) +\sum_{j=i}^{N-1}f_{l,j}(X_j, Y_{l,j+1}, Z_{l,j}) \Dt},\\
\Dt  Z_{l,i} &= \Espi{ ( g(X_N) +\sum_{j=i+1}^{N-1} f_{l,j}(X_j, Y_{l,j+1},Z_{l,j})\Dt ) \Delta W_i},
\end{align*}
for $i \in \{0, \dots ,  N-1\}$, $l\in\{1,2\}.$

To allow the driver $f_{1,i}$ to depend on the clouds of simulations (necessary in the analysis), we require that it is measurable w.r.t. $\cF_T$ instead of $\cF_{t_i}$ as usually.
\begin{proposition}\label{prop: apriori}
Assume that \HG\ and \HX hold.
Moreover, for each $i \in \{0, \ldots, N-1\}$, assume that $f_{1,i}(X_i,Y_{1,i+1},Z_{1,i}) \in \L_2(\cF_T)$ and
$ f_{2}$ satisfies \HF\ with constants $L_{f_2}$ and $C_{f_2}$.
Then, for any  $\gamma\in (0,+\infty)$
satisfying $6q(\Dt+\frac{ 1 }{\gamma})L^2_{  f_{2}} \leq 1$, we have for $0\leq i < N$
\begin{align}
&|Y_{1,i}- Y_{2,i}|^2 \Gamma_i+ \sum_{j=i}^{N-1} \Dt \Espi{|Z_{1,j}- Z_{2,j}|^2} \Gamma_j \nonumber\\
&
\leq 3\cprop \left(\frac{ 1 }{\gamma }+\Dt\right) \sum_{j=i}^{N-1}
\Dt \Espi{|f_{1,j}(X_j,Y_{1,j+1},Z_{1,j})- f_{2,j}(X_j,Y_{1,j+1},Z_{1,j})|^2} \Gamma_j
,
\label{eq:prop21}
\end{align}
where
$\Gamma_i :=  (1+\gamma \Dt)^i$ and $\cprop:={2} q+(1+T)e^{T/2}$.
\end{proposition}

% \section*{Acknowledgments} is not working
\vspace{1cm}
{\bf{Acknowledgments}}
\vspace{0.3cm}

The authors are gratefull to the three anonymous referees whose remarks have helped us to improve the final version of the paper.

%\bibliographystyle{siam.bst}
%\bibliography{Stratification_LSMDP_bib_GPUs}

\end{document}